\documentclass{article}
\usepackage{a4,multicol}
\usepackage[latin1]{inputenc}
\usepackage[leqno]{amsmath}
\usepackage{amsfonts}
\usepackage{amsthm,amssymb,amstext,amsxtra}
\usepackage{mathbbol}
\usepackage{graphicx}%,picins
\usepackage{mathrsfs}
\usepackage{color}
\usepackage{verbatim}
\usepackage{caption}
\usepackage{epsfig}
\newcommand{\floor}[1]{\left\lfloor #1 \right\rfloor}
\newcommand{\ceil}[1]{\left\lceil #1 \right\rceil}

\newcommand{\NN}{\mathbb{N}}

\newcommand{\ZZ}{\mathbb{Z}}
\newcommand{\RR}{\mathbb{R}}
\newcommand{\PP}{\mathbb{P}}
\newcommand{\KK}{\mathbb{K}}
\newcommand{\col}{\text{col}}
\newcommand{\row}{\text{row}}
\newcommand{\defn}[1]{{\it \color{red} #1}}
\newcommand{\twovec}[2]{\left(\begin{smallmatrix}#1\\#2\end{smallmatrix}\right)}
\newcommand{\mtwovec}[2]{\begin{pmatrix}#1\\#2\end{pmatrix}}
\newcommand{\threevec}[3]{\left(\begin{smallmatrix}#1\\#2\\#3\end{smallmatrix}\right)}
\newcommand{\mthreevec}[3]{\begin{pmatrix}#1\\#2\\#3\end{pmatrix}}
\newcommand{\interval}[3]{#1|_{[#2,#3]}}
\newcommand{\cone}{\mathsf{cone}}
\newcommand{\txtones}{\mathsf{ones}}
\newcommand{\txtlength}{\mathsf{length}}

\newcommand{\ppo}{\Pi^\circ}

\newcommand{\pS}[2]{\ppo_{\downarrow,{#1},{#2}}}
\newcommand{\pC}[2]{\ppo_{\rightarrow,{#1},{#2}}}
\newcommand{\gD}[2]{d^\downarrow_{#1,#2}}
\newcommand{\gR}[2]{d^\rightarrow_{#1,#2}}

\newcommand{\vpiper}[4]{\mathcal{V}_{#1,#2,#3}^{#4}}
\newcommand{\vpipe}[3]{\mathcal{V}_{#1,#2}^{#3}}
\newcommand{\hpiper}[4]{\mathcal{H}_{#1,#2,#3}^{#4}}
\newcommand{\hpipe}[3]{\mathcal{H}_{#1,#2}^{#3}}

\newcommand{\nwsearrow}{\mathrel{\text{$\nwarrow$\llap{$\searrow$}}}}
\newcommand{\neswarrow}{\mathrel{\text{$\nearrow$\llap{$\swarrow$}}}}
\newcommand{\nwneseswarrow}{\mathrel{\text{$\neswarrow$\llap{$\nwsearrow$}}}}

\newcommand{\diagrefl}{\nwsearrow\!}
\newcommand{\orgrefl}{\nwneseswarrow\!}

\newcommand{\red}[1]{\underline{#1}}
\newcommand{\ones}{\mathbb{1}}
\newcommand{\period}[1]{\text{period}#1}

\newcommand{\swap}[2]{\text{swap}(#1,#2)}

\newcommand{\set}[2]{\left\{#1\;\middle|\;#2\right\}}
\newcommand{\define}{\mathrel{\mathop:}=}
\newcommand{\conv}{\operatorname{conv}}
\renewcommand{\mod}{\mathrel{\;\operatorname{ mod }\;}}
\renewcommand{\div}{\mathrel{\;\operatorname{ div }\;}}

\newcommand{\rAr}[0]{\Rightarrow}

\newcommand{\lrAr}[0]{\Leftrightarrow}

\newtheorem{Lemma}{Lemma}[section]
\newtheorem{Cor}[Lemma]{Corollary}
\newtheorem{Thm}[Lemma]{Theorem}

\theoremstyle{definition}

\newtheorem{fact}{Fact}

\theoremstyle{remark}

\numberwithin{equation}{section}
\title{Staircases in $\ZZ^2$}
\author{
Felix Breuer\thanks{Institut f\"ur Mathematik, Freie Universit\"at Berlin, E-mail: felix.breuer@fu-berlin.de\newline Research supported by the Deutsche Forschungsgemeinschaft within the research training group 'Methods for Discrete Structures' (GRK 1408).}
\and 
Frederik von Heymann\thanks{Faculty Electrical Engineering, Mathematics and Computer Science, Delft University of Technology,
E-mail: F.J.vonHeymann@tudelft.nl\newline During the largest part of this research funded by the DFG Emmy Noether program (HA 4383/1) and the Freie Universit\"at Berlin.}}

\begin{document}
\maketitle

%%%%%%%%%%%%%%%%%%%%%%%%%%%%%%%%%%%%%
%                                   %
%             section 1             %
%                                   %
%%%%%%%%%%%%%%%%%%%%%%%%%%%%%%%%%%%%%

\begin{abstract}
A staircase in this paper is the set of points in $\ZZ^2$ below a given rational line in the plane that have Manhattan Distance less than 1 to the line. Staircases are closely related to Beatty and Sturmian sequences of rational numbers. Connecting the geometry and the number theoretic concepts, we obtain three equivalent characterizations of Sturmian sequences of rational numbers, as well as a new proof of Barvinok's Theorem in dimension two, a recursion formula for Dedekind-Carlitz polynomials and a partially new proof of White's characterization of empty lattice tetrahedra. Our main tool is a recursive description of staircases in the spirit of the Euclidean Algorithm.
\end{abstract}

\section{Introduction}

Motivated by the study of lattice points inside polytopes, in this paper we seek to understand the set of lattice points ``close'' to a rational line in the plane. To this end we define a staircase in the plane to be the set of lattice points in the half-plane below a rational line that have Manhattan Distance less than 1 to the line. We prove several properties of these point sets, most importantly we show that they have a recursive structure that is reminiscent of the Euclidean Algorithm.

Not surprisingly, staircases are closely related to the Beatty and Sturmian sequences defined in number theory (see \cite{S76,FMT78,PS90,Bry02}), i.e.\ to sequences of the form $\left(\floor{\frac{b}{a}n}-\floor{\frac{b}{a}(n-1)}\right)_n$ for $a,b\in\NN$ with $\gcd(a,b)=1$. We show several elementary properties of these sequences from a geometric point of view. To our knowledge such a geometric approach to these sequences is not available in the prior literature. Our observations lead to three characterizations of these sequences (Theorem~\ref{thm:equivalences}). One of these is known (see \cite{GLL78,Fraenkel05}) while the other two seem to be new.

We conclude the paper by giving several applications of our findings. Firstly, we give a new proof of a theorem by Barvinok in dimension 2. Barvinok's Theorem states that the generating function of the lattice points inside a rational simplicial cone can be written as a short rational function. While Barvinok uses a signed decomposition of the cone into unimodular cones to achieve this result, we partition the cone into sets that have a short representation. 

Secondly, these ideas can also be used to give a recursion formula for Dedekind-Carlitz polynomials. These are polynomials of the form $\sum_{k=1}^{a-1}x^{k-1}y^{\floor{\frac{b}{a}k}}$ or, equivalently, generating functions of the lattice points inside the open fundamental parallelepipeds of cones of the form $\cone\left(\twovec{0}{-1},\twovec{a}{b}\right)$. Our recursion formula answers a question from \cite{BeckHaase08}. 

Finally, we simplify ideas from \cite{Scarf85} and \cite{Reznick06} to give a partially new proof of White's Theorem, which characterizes three-dimensional lattice simplices that contain no lattice points except their vertices.

This paper is organized as follows. In Section~\ref{sec:theproblem} we give definitions of and present elementary facts about the objects we study. While Section~\ref{sec:theproblem} is rather concise, we elaborate more in Section~\ref{sec:geometricidea}, where staircases are examined from a geometric point of view. Most importantly we explain the recursive structure of staircases in Lemma~\ref{lem:recursion}, \ref{lem:recursion-parallelepiped} and \ref{lem:recursion-triangle}. In Section~\ref{sec:characterization} we motivate the three characterizations of Sturmian sequences before summarizing them in Theorem~\ref{thm:equivalences}. Section~\ref{sec:theproof} is devoted to the proof of this theorem. In Section~\ref{sec:short-representations} we apply Lemma~\ref{lem:recursion-triangle} to give a new proof of Barvinok's Theorem in dimension 2 and in Section~\ref{sec:dedekind-carlitz} we use Lemma~\ref{lem:recursion-parallelepiped} to give a recursion formula for Dedekind-Carlitz sums. We conclude the paper in Section~\ref{whitethm} by giving a partially new proof of White's Theorem.

%%% Local Variables: 
%%% mode: latex
%%% TeX-master: "Staircases_in_Z2"
%%% End: 

% \input{staircases_and_related}
%%%%%%%%%%%%%%%%%%%%%%%%%%%%%%%%%%%%%
%                                   %
%             section 2             %
%                                   %
%%%%%%%%%%%%%%%%%%%%%%%%%%%%%%%%%%%%%
\section{Staircases and Related Sequences}\label{sec:theproblem}

In this section we give the basic definitions we are going to work with. In particular we introduce staircases, which are the main geometric objects we will analyze. Then we will define some related sequences of integers and state basic facts about them and their connection to staircases.
We elaborate on the geometric point of view and give additional examples in Section~\ref{sec:geometricidea}.

\vspace{1em}
Before we introduce staircases, here are some preliminaries: For any real number $r\in\RR$ we define the \defn{integral part} $\floor{r}:=\max\set{z\in\ZZ}{z\leq r}$ of $r$. The \defn{fractional part} $\{r\}$ of $r$ is then defined by $r=\floor{r} + \{r\}$. Given $0<a,b\in\NN$ there exist unique integers $(b\div a)$ and $(b\mod a)$ such that $b=(b\div a)\cdot a +(b\mod a)$ and $0\leq b\mod a < a$. Using these two functions we can write $\floor{\frac{b}{a}}=b \div a$ and $\{\frac{b}{a}\}=\frac{b\mod a}{a}$. We are going to use these two notations interchangeably.

\vspace{1em}
Given $A,B\subset\RR^2$ and $v\in\RR^2$, we define $A+B:=\set{a+b}{a\in A,b\in B}$ and ${-A:=\set{-a}{a\in A}}$ and we use the abbreviations $A-B:= A+(-B)$ and ${A+v}={A+\{v\}}$. We will refer to $A+B$ as the \defn{Minkowski sum} of $A$ and $B$. The difference of sets is denoted with $A\setminus B:=\set{a\in A}{a\not\in B}$.
A \defn{lattice point} is an element of $\ZZ^2$. An affine (linear) \defn{lattice transformation} of the plane is an affine (linear) automorphism of the plane that maps $\ZZ^2$ bijectively onto $\ZZ^2$. A vector $v\in\ZZ^d$ is \defn{primitive} if $\gcd(v_1,\ldots,v_d)=1$ or, equivalently, if $\ZZ^d\cap\conv(0,v)=\{0,v\}$. 

\vspace{1em}
Now, what is a staircase? Let $L$ be an oriented rational line in the plane. Then $L$ defines a positive half-space $H$. The task is to describe the lattice points in $H$ that are close to $L$ in the sense that they have distance $< 1$ to the line in the Manhattan metric. Equivalently we consider those points $x\in \ZZ^2\cap H$ from which we can reach a point in the other half-space by a single horizontal or vertical step of unit length. Such a set of points we call a staircase. See Figure~\ref{fig:staircase-example} for two examples. Note that it is sufficient to depict the staircase only under a primitive vector generating the line (in the first example the vector $(3,8)$), as after that (and before that) the same pattern of points is repeated.

We do not require $L$ to pass through the origin (or any other element of $\ZZ^2$), contrary to what Figure~\ref{fig:staircase-example} might suggest. But we will see later that we can get all the information we want by looking only at lines through the origin. Also, without loss of generality we will restrict our attention to lines with positive slope, as negative slopes will give us, up to mirror symmetry, the same sets.

\begin{figure}[ht]
\begin{center}
\input{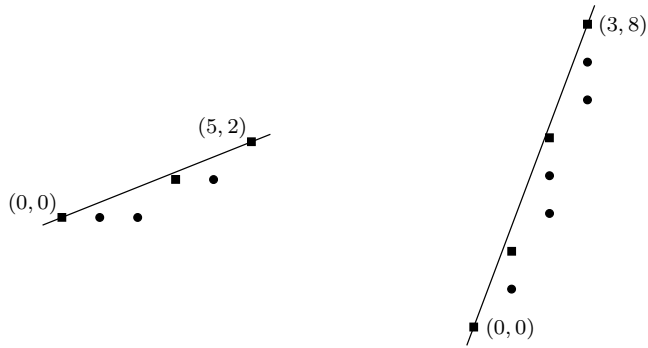}
\caption{\label{fig:staircase-example}A part of the staircases $S_{5,2}$ and $S_{3,8}$. Corners are shown as boxes.}
\end{center}
\end{figure}

Let $0<a,b\in\NN$ with $\gcd(a,b)=1$ and let $r\in\RR$. These parameters define the line $L_{a,b,r}=\set{x\in\RR^2}{x_2=\frac{b}{a}\,x_1+r}$. We denote the closed half-spaces below and above that line by $H^+_{a,b,r}$ and $H^-_{a,b,r}$, respectively. 
Formally we define for any $\sigma\in\{+1,-1\}$ the \defn{half-space} $H^\sigma_{a,b,r}$ as
\begin{eqnarray*}
H^\sigma_{a,b,r}&=&\set{x\in\RR^2}{0 \leq \sigma (\frac{b}{a}\,x_1 - x_2 + r) }.
\end{eqnarray*}
Most of the time we will use $+$ to represent $+1$ and $-$ to represent $-1$ and write $\sigma\in\{+,-\}$ for short. Also if $\sigma=+$ and/or $r=0$ we will omit these parameters and write $H_{a,b}$ for $H^{+1}_{a,b,0}$, and similarly for the symbols introduced below. The case $\sigma=+$ and $r=0$ is of the largest interest to us, as all other cases can be reduced to this one. 

We will give a precise formulation and proof of this in Lemma~\ref{lem:translation}. Although the proof does not require additional tools and could already be given here, we will pursue the connection between certain sequences and our point sets first, and postpone all observations that are purely concerned with the point sets to Section~\ref{sec:geometricidea}. So let's start defining these sets properly, to make it more clear what we are talking about.

\vspace{1em}
The following definitions are illustrated in Figure~\ref{fig:Fact1}. The lattice points in $H^\sigma_{a,b,r}$ that are at distance less than 1 from the line in vertical and horizontal direction, respectively, are
\begin{eqnarray*}
\vpiper{a}{b}{r}{\sigma} &=& 
\ZZ^2\cap H^\sigma_{a,b,r}\setminus (H^\sigma_{a,b,r}-\sigma e_2) \\ &=& 
\set{z\in\ZZ^2}{0 \leq \sigma(\frac{b}{a} z_1 - z_2 + r) < 1} \\
\hpiper{a}{b}{r}{\sigma} &=& 
\ZZ^2\cap H^\sigma_{a,b,r}\setminus (H^\sigma_{a,b,r}+\sigma e_1) \\ &=& 
\set{z\in\ZZ^2}{0 \leq \sigma(\frac{b}{a} z_1 - z_2 + r) < \frac{b}{a}}. 
\end{eqnarray*}

Using this notation we now define the \defn{staircase} $S_{a,b,r}^\sigma$ to be the set of points that are at distance less than~1 from the line in horizontal \emph{or} vertical direction, and we define the \defn{corners} $C_{a,b,r}^\sigma$ to the be the lattice points that are at distance less than~1 in horizontal \emph{and} vertical direction:
\begin{eqnarray*}
S^\sigma_{a,b,r} &=& \vpiper{a}{b}{r}{\sigma}\cup \hpiper{a}{b}{r}{\sigma}, \\
C^\sigma_{a,b,r} &=& \vpiper{a}{b}{r}{\sigma}\cap \hpiper{a}{b}{r}{\sigma}.
\end{eqnarray*}

In other words
\begin{eqnarray*}
S_{a,b}^\sigma & = & \set{z\in\ZZ^2}{z\in H_{a,b}^\sigma\text{ but }z-\sigma e_1\not\in H_{a,b}^\sigma\text{ or }z+\sigma e_2\not\in H_{a,b}^\sigma} \\
C_{a,b}^\sigma & = & \set{z\in\ZZ^2}{z\in H_{a,b}^\sigma\text{ but }z-\sigma e_1\not\in H_{a,b}^\sigma\text{ and }z+\sigma e_2\not\in H_{a,b}^\sigma}.
\end{eqnarray*}

See Figure~\ref{fig:staircase-example} and also Figure~\ref{fig:Fact1}. Clearly $C_{a,b}\subset S_{a,b}$. For any set $A\subset\ZZ^2$ and any $x\in\ZZ$ we call the set $\col_x(A)=\{(x,y)\in A\}$ a \defn{column} of $A$ and for $y\in\ZZ$ we call the set $\row_y(A)={\{(x,y)\in A\}}$ a \defn{row} of $A$. For any $0<a,b\in\NN$, every row and every column of $S_{a,b}$ contains at least one point and every row and every column of $C_{a,b}$ contains at most one point. 

The sequence $(|\col_x(S_{a,b})|)_{x\in\ZZ}$ is called the \defn{column sequence} of $S_{a,b}$, the sequence $(|\row_y(C_{a,b})|)_{y\in\ZZ}$ is called the \defn{row sequence} of $C_{a,b}$ and so on.

\vspace{1em}
In the following we summarize some basic facts about staircases. We omit the proofs as they are easy enough to do and would slow the pace of this section without giving the reader further insights. The reader may find it instructive, however, to check the validity of these facts by looking at examples such as those given in the figures of this section.

\begin{fact}\label{fact:pipes}
For all $0<a,b\in\NN$ and $\sigma\in\{+,-\}$
\begin{eqnarray*}
a \geq b \quad \lrAr & \hpipe{a}{b}{\sigma}\subset \vpipe{a}{b}{\sigma}  & \lrAr \quad \forall x: |\col_x(S_{a,b}^\sigma)| = 1 \quad \lrAr \quad \forall y: |\row_y(C_{a,b}^\sigma)| = 1,\\
 a \leq b \quad \lrAr & \hpipe{a}{b}{\sigma}\supset \vpipe{a}{b}{\sigma}  & \lrAr \quad \forall y: |\row_y(S_{a,b}^\sigma)| = 1 \quad \lrAr \quad \forall x: |\col_x(C_{a,b}^\sigma)| = 1.
\end{eqnarray*}
\end{fact}

In the former case we call $S_{a,b}^\sigma$ \defn{flat} and in the latter case we call $S_{a,b}^\sigma$ \defn{steep}, see Figure~\ref{fig:Fact1}. Note that this implies $S_{a,b}^\sigma = \vpipe{a}{b}{\sigma}$ for flat and $S_{a,b}^\sigma = \hpipe{a}{b}{\sigma}$ for steep staircases. 

\begin{figure}[ht]
\begin{center}
 \input{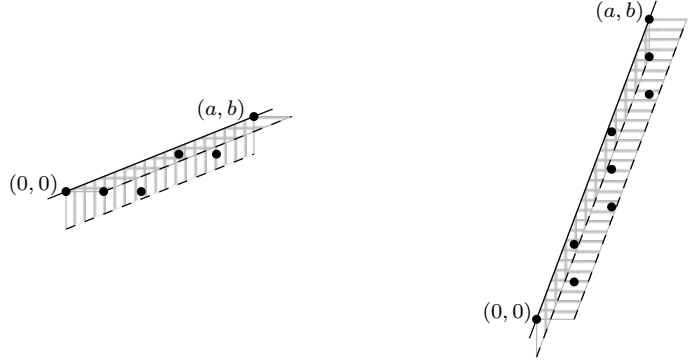}
\caption{\label{fig:Fact1}This figure shows the flat staircase $S_{5,2}$ and the steep staircase $S_{3,8}$ from Figure~\ref{fig:staircase-example} together with parts of the corresponding sets $\hpipe{a}{b}{}$ and $\vpipe{a}{b}{}$. This illustrates Fact~\ref{fact:pipes}.}
\end{center}
\end{figure}

\begin{fact}\label{fact:toppoint}
For all $n\in\ZZ$ the topmost point of $\col_n(S_{a,b})$ is $(n,\floor{\frac{b}{a}n})$.
\end{fact}

If $S_{a,b}$ is flat, we have seen in Fact~\ref{fact:pipes} that for every $n\in\ZZ$ the set $\col_n(S_{a,b})$ contains exactly one element, so all elements of $S_{a,b}$ have the form $(n,\floor{\frac{b}{a}n})$. If $S_{a,b}$ is steep, this is only true for the corners.

\vspace{1em}
This description allows us to compute the difference in height between the topmost points in consecutive columns of $S_{a,b}$. For all $0<a,b\in\NN$ we define the sequence $B_{a,b}=(B_{a,b}(n))_{n\in\ZZ}$ by
\begin{eqnarray}
\label{eqn:Beatty-1}\nonumber
B_{a,b}(n) & := & \floor{\frac{b}{a}n}-\floor{\frac{b}{a}(n-1)} \\ & = &
\frac{b}{a}+\left\{\frac{b}{a}(n-1)\right\}-\left\{\frac{b}{a}n\right\}.
\label{eqn:Beatty-2}
\end{eqnarray}

\begin{fact}\label{fact:interpretation}
 If $a\leq b$ (i.e.\ $S_{a,b}$ is steep) then
\[
 |\col_n(S_{a,b})|=B_{a,b}(n)
\]
and if $b\leq a$ (i.e.\ $S_{a,b}$ is flat) then 
\[
 |\col_n(C_{a,b})|=B_{a,b}(n),
\]
in particular $B_{a,b}$ is a $0,1$-sequence in this case.
\end{fact}

The sequence $B_{a,b}$ is the key to connect the geometric description of ``points close to a line'' with notions from number-theory. 

\vspace{1em}
The sequences $(\floor{\frac{b}{a}n})_{n\in\NN}$, $(B_{a,b}(n))_{n\in\NN}$ and $(\red{B}_{a,b}(n))_{n\in\NN}$ (see below) are known as the \defn{characteristic sequence}, the \defn{Beatty sequence} and the \defn{Sturmian sequence} of $\frac{b}{a}$, respectively. 

\vspace{1em}
These sequences are well studied in number theory, see~\cite{S76,FMT78,PS90,Bry02} for surveys, \cite{B93} for historical remarks and~\cite{Fraenkel05} for a discussion about the names of the sequences. However, only the characteristic sequences of irrational numbers are non-trivial from the point of view of number theory. They also appear in geometry, see Section~\ref{whitethm}. 

\vspace{1em}
For the rest of the section we will establish some definitions connected with the above sequences, and some basic properties of $B_{a,b}$, relating them to the staircases.

\vspace{1em}
In this spirit, instead of working with $(\floor{\frac{b}{a}n})_{n\in\NN}$, $(B_{a,b}(n))_{n\in\NN}$ and $(\red{B}_{a,b}(n))_{n\in\NN}$, we will deal with $(\floor{\frac{b}{a}n})_{n\in\ZZ}$, $(B_{a,b}(n))_{n\in\ZZ}$ and $(\red{B}_{a,b}(n))_{n\in\ZZ}$, respectively. This is due to the fact that we look at staircases of lines, not of rays.

\vspace{1em}
Now let's define what a Sturmian sequence is. A sequence $s=(s_n)_{n\in\ZZ}$ of integers $s_n\in\ZZ$ is called \defn{balanced} (at $k$) if $s_n\in\{k,k+1\}$ for all $n\in\ZZ$. If $s$ is balanced at $k$, we can define a $0,1$-sequence $\red{s}=(\red{s}_n)_{n\in\ZZ}$, which we call the \defn{reduced} sequence, by 
\[
\red{s}_n=s_n-k.
\]
 Note that if $s=(c)_{n\in\ZZ}$ is constant, $s$ is balanced at both $c$ and $c-1$. In this case $\red{s}$ is defined with respect to $c$, i.e.\ $\red{s}$ is constant 0.

\begin{Lemma}
\label{Lem:beatty-is-balanced}
$B_{a,b}$ is balanced at $\floor{\frac{b}{a}}$. If $\frac{b}{a}\in\ZZ$, then $B_{a,b}(n)=\frac{b}{a}$ for all $n\in\ZZ$.
\end{Lemma}
\begin{proof}
By~(\ref{eqn:Beatty-2}) we know that $|B_{a,b}(n)-\frac{b}{a}|<1$ and by definition $B_{a,b}\in\ZZ$. If $\frac{b}{a}\in\ZZ$, then the fractional parts in~(\ref{eqn:Beatty-2}) are both $0$, and thus the second statement is also true.
\end{proof}

So Sturmian sequences $(\red{B}_{a,b}(n))_{n\in\NN}$ are well-defined.
Furthermore, we now know that only two different integers appear in $B_{a,b}$, and that $\red{B}_{a,b}$ tells us in which positions the larger integer of the two appears.

Given our geometric interpretation of $B_{a,b}$ from Fact~\ref{fact:interpretation} this means that a steep staircase has columns of only two different lengths and the reduced sequence $(\red{B}_{a,b}(n))_{n\in\NN}$ encodes which columns are long and which columns are short. We will return to the concept of reduction in Section~\ref{sec:geometricidea}.

\vspace{1em}
To make the connection between the sequence $B_{a,b}$ and the point set $S_{a,b}$ more transparent, we introduce some more notation.

\vspace{1em}
We will transfer the idea that all we want to know about the staircase can be found under a primitive integer vector in the line $L_{a,b}$ into the language of sequences.

For any sequence $s=(s_n)_{n\in\ZZ}$ we say that $s$ is \defn{periodic} with period $a\in\NN$ if $(s_{n+a})_{n\in\ZZ}=(s_n)_{n\in\ZZ}$. We say that $a$ is the \defn{minimal period} of $s$ if there is no period $a'\in\NN$ of $s$ with  $a'<a$ and write $\PP(s)$ for the minimal period of $s$. By (\ref{eqn:Beatty-2}), if $\gcd(a,b)=1$, then $B_{a,b}$ is periodic with minimal period $a$.

For a periodic sequence $s$ we define $\period(s)=(s_n)_{0\leq n < \PP(s)}$. If $s$ is a periodic $0,1$-sequence, we write $\ones(s)$ for the number of ones in $\period(s)$. We will frequently represent $s$ by the $\PP(s)$-tuple $\period(s)$. 

As $B_{a,b}$ describes the differences of the maximal heights in adjacent columns of $S_{a,b}$, these differences, accumulated between $0$ and $a-1$, must sum up to $b$. 

\vspace{1em}
We summarize the above observations into
\begin{fact}\label{fact:Beatty-period}
If $0<a,b\in\NN$ and $\gcd(a,b)=1$, then 
\[
 \PP(B_{a,b})=a\quad\text{ and }\sum_{0\leq n < a}B_{a,b}(n)=b.
\]
In particular if $a>b$ (and thus $S_{a,b}$ is flat), then $\ones(B_{a,b})=b$.
\end{fact}

To be more flexible when talking about parts of staircases respectively Beatty-sequences, we define the following.
Given a sequence $s=(s_n)_{n\in\ZZ}$, a finite subsequence of the form 
\[
\interval{s}{x_0}{x_1} :=(s_n)_{x_0\leq n \leq x_1}\text{ for some }x_0\leq x_1\in\ZZ
\]
will be called an \defn{interval}. The number of elements $x_1-x_0+1$ of $\interval{s}{x_0}{x_1}$ we will call the length of the interval and we will denote it by $\txtlength(\interval{s}{x_0}{x_1})$. If $s$ is a $0,1$\mbox{-}sequence, we will denote the number of ones in an interval $\interval{s}{x_0}{x_1}$ by $\txtones(\interval{s}{x_0}{x_1})$. 

\vspace{1em}
In Fact~\ref{fact:Beatty-period} we summed over the interval $\interval{B_{a,b}}{0}{a-1}$. But because of the periodicity of $S_{a,b}$ and $B_{a,b}$ we see that we could have used any interval of length $a-1$. So for any fixed $i\in\ZZ$ the sequence $B_{a,b}(n+i))_{n\in\ZZ}$ also describes $S_{a,b}$. This gives rise to the following definition:

We say that sequences $s=(s_n)_{n\in\ZZ}$ and $s'=(s'_n)_{n\in\ZZ}$ are identical \defn{up to shift} if there exists an $i\in\ZZ$ with $(s_{n+i})_{n\in\ZZ}=(s'_n)_{n\in\ZZ}$, in symbols $s\equiv s'$. Our goal in Section~\ref{sec:characterization} will be to characterize Sturmian sequences up to shift.

%%% Local Variables: 
%%% mode: latex
%%% TeX-master: "Staircases_in_Z2"
%%% End: 

% \input{geometric_obs}
%%%%%%%%%%%%%%%%%%%%%%%%%%%%%%%%%%%%%
%                                   %
%             section 3             %
%                                   %
%%%%%%%%%%%%%%%%%%%%%%%%%%%%%%%%%%%%%
\section{Geometric Observations}\label{sec:geometricidea}

In this section we develop some properties of staircases and their related sequences from a geometric point of view. The most important operation on staircases is for us the reduction, which we turn to in the latter half of this section. We start with some more elementary operations.

\vspace{1em}
Throughout this section let  $0<a,b\in\NN$ such that $\gcd(a,b)=1$ and let $\sigma\in\{+,-\}$.

\paragraph{Elementary Properties of Staircases.} As we have already mentioned (and used) before, all staircases with a given slope, regardless whether it's the one above or below the line, are translates of each other. Hence they yield the same step sequence up to shift. Before we finally prove this, we state an elementary lemma.

\begin{Lemma}
\label{lem:dist-to-line}
Let $r\in\RR$.
The line $L_{a,b,r}$ contains a lattice point if and only if $r=\frac{k}{a}$ for some $k\in\ZZ$.
\end{Lemma}

\begin{proof}
Without loss of generality we can assume $-1<r\leq 0$.
For any point $z\in\ZZ^2$ the vertical distance to the line $L_{a,b}$ is $\frac{b}{a}z_1 -z_2 = \frac{k}{a}$ for some $k\in\ZZ$. So if $r\not=\frac{k}{a}$ for any $k\in\ZZ$, then $L_{a,b,r}$ cannot contain a lattice point. 

That $r=\frac{k}{a}$ is sufficient for the existence of a lattice point follows directly from the extended Euclidean Algorithm. It can also be shown with this geometric argument: $L_{a,b,r}$ can contain at most one lattice point $z$ with $0\leq z_1 < a$, for if there were two distinct lattice points with this property then $\gcd(a,b)\not=1$. On the other hand $\vpipe{a}{b}{}\cap ([0,a)\times \RR)$ contains exactly $a$ lattice points, one in each column. Only the lines $L_{a,b,-\frac{k}{a}}$ with $0\leq k\leq a-1$ can intersect $\vpipe{a}{b}{}$. So each of them has to contain at least one lattice point.
\end{proof}

\begin{Lemma}
\label{lem:translation}
For every $0<a,b\in\NN$ and $r\in\RR$
\begin{enumerate}
\renewcommand{\theenumi}{\arabic{section}.\arabic{Lemma}.\arabic{enumi}}
\renewcommand{\labelenumi}{\arabic{enumi}.}
\item $S_{a,b,r} = S_{a,b} + v$ and $C_{a,b,r} = C_{a,b} + v$ for some $v\in\ZZ^2$ and \label{lem:translation1}
\item  $S^-_{a,b} = S_{a,b} + v$ and $C^-_{a,b} = C_{a,b} + v$ for some $v\in\ZZ^2$.\label{lem:translation2}
\end{enumerate}
\end{Lemma}

\begin{proof}
\emph{1.} By Lemma~\ref{lem:dist-to-line}, if $\frac{k}{a}\leq r < \frac{k+1}{a}$ then $S_{a,b,r}=S_{a,b,\frac{k}{a}}$ and $C_{a,b,r}=C_{a,b,\frac{k}{a}}$. Hence we can assume without loss of generality $r=\frac{k}{a}$, so the line $L_{a,b,r}$ contains a lattice point $v=(v_1,v_2)$ with $v_2=\frac{b}{a}v_1+r$. Then
\begin{eqnarray*}
z\in H_{a,b,r} - v & \lrAr & z_2 + v_2 \leq \frac{b}{a}(z_1 + v_1) +r \\
& \lrAr & z_2  \leq \frac{b}{a}z_1 \lrAr z\in H_{a,b}.
\end{eqnarray*}
This implies the first claim.

\emph{2.} 
By Lemma~\ref{lem:dist-to-line} there is no lattice point $v'$ with $\frac{a-1}{a}<\frac{b}{a}v'_1 -v'_2<1$, so
\begin{eqnarray*}
\vpipe{a}{b}{-} &=&
\set{z\in\ZZ^2}{0\leq -(\frac{b}{a}z_1 - z_2) \leq \frac{a-1}{a}}.
\end{eqnarray*}

Also by Lemma~\ref{lem:dist-to-line}, there exists a lattice point $v$ with $\frac{b}{a}v_1 -v_2=-\frac{a-1}{a}$ and for this point $v$
\begin{eqnarray*}
\vpipe{a}{b}{-} - v&=&
\set{z\in\ZZ^2}{0\leq -(\frac{b}{a}(z_1+v_1) - (z_2+v_2)) \leq \frac{a-1}{a}}\\ &=&
\set{z\in\ZZ^2}{0\leq -(\frac{b}{a}z_1 - z_2 -\frac{a-1}{a}) \leq \frac{a-1}{a}}\\ &=&
\set{z\in\ZZ^2}{\frac{a-1}{a}\geq \frac{b}{a}z_1 - z_2 \geq 0}.
\end{eqnarray*}
Applying the first observation again, we obtain
$$\vpipe{a}{b}{-} - v = \vpipe{a}{b}{+}.$$
A similar argument shows $\hpipe{a}{b}{-} - v = \hpipe{a}{b}{+}$ for a suitable $v$. Now, because of Fact~\ref{fact:pipes}, $S_{a,b}^\sigma=\vpipe{a}{b}{\sigma}$ and $C_{a,b}^\sigma=\hpipe{a}{b}{\sigma}$ or $S_{a,b}^\sigma=\hpipe{a}{b}{\sigma}$ and $C_{a,b}^\sigma=\vpipe{a}{b}{\sigma}$, depending on whether $a>b$ or $a<b$, where $\sigma\in\{+,-\}$. Therefore the above calculations imply \ref{lem:translation2}.
\end{proof}

\vspace{1em}
The previous operations translated the staircases by an integral vector. Now we will introduce some other useful operations. We denote the reflection at the main diagonal by $\diagrefl\,$ and the reflection at the origin by $\orgrefl\,$, i.e.\ we define $\diagrefl(x,y)=(y,x)$ and $\orgrefl(x,y)=(-x,-y)$. Note that both induce involutions, i.e.\ self-inverse bijective maps, on sets of lattice points; so we understand a set of lattice points if and only if we understand its reflection. The effect of these two reflections on staircases is illustrated with an example in Figures~\ref{fig:diag-reflection} and~\ref{fig:org-reflection} and formalized in Lemmas~\ref{lem:diag-reflection} and~\ref{lem:org-reflection}. 

\begin{figure}[ht]
\begin{center}
 \includegraphics[scale=0.3]{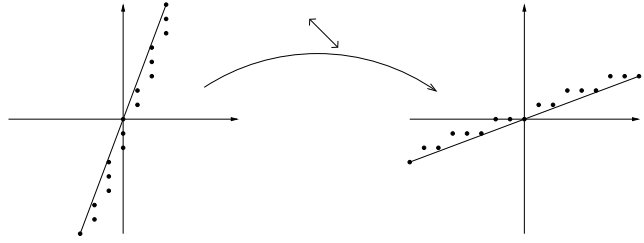}
\caption{\label{fig:diag-reflection}The reflection at the main diagonal swaps numerator and denominator of a staircase and places the points on the opposite side of the line. Here we see $\diagrefl S_{3,8}=S^-_{8,3}$.}
\end{center}
\end{figure}

\begin{figure}[ht]
\begin{center}
 \includegraphics[scale=0.3]{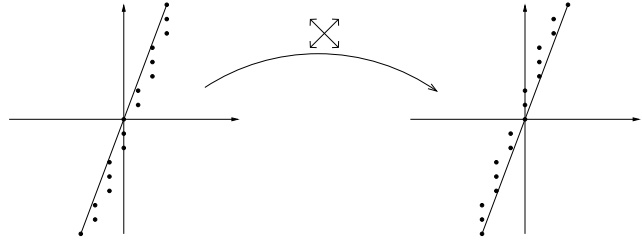}
\caption{\label{fig:org-reflection} The reflection at the origin transforms a staircase below the line into a staircase above the line and vice versa. Here we see $\orgrefl S_{3,8}= S^-_{3,8}$. Notice how the column sequence is reversed!}
\end{center}
\end{figure}

\begin{Lemma}
\label{lem:diag-reflection}
$\diagrefl S_{a,b}^\sigma = S_{b,a}^{-\sigma}$ and $\diagrefl C_{a,b}^\sigma = C_{b,a}^{-\sigma}$.
\end{Lemma}

In other words, reflection at the main diagonal swaps numerator and denominator of the slope and places the points on the opposite side of the line. See Figure~\ref{fig:diag-reflection}.

\begin{proof} We compute
\begin{eqnarray*}
 (x,y)\in H_{a,b}^\sigma &\lrAr& \sigma y\leq \sigma(\frac{b}{a}x) \quad\lrAr\quad \sigma x \geq \sigma (\frac{a}{b}y)\\
&\lrAr&-(\sigma x)\leq -\sigma(\frac{a}{b}y)\quad\lrAr\quad (y,x)\in H_{b,a}^{-\sigma}\\
&\lrAr& \diagrefl(x,y)\in H_{b,a}^{-\sigma}.
\end{eqnarray*}
This implies both 
\begin{eqnarray*}
z-\sigma e_1\not\in H^\sigma_{a,b} & \lrAr & \diagrefl (z)+(-\sigma)e_2\not\in H^{-\sigma}_{b,a} \quad \text{and}\\ 
z+\sigma e_2\not\in H^\sigma_{a,b} & \lrAr & \diagrefl (z)-(-\sigma)e_1\not\in H^{-\sigma}_{b,a}.
\end{eqnarray*}
All three equivalences taken together give $\diagrefl S_{a,b}^\sigma = S_{b,a}^{-\sigma}$ and $\diagrefl C_{a,b}^\sigma = C_{b,a}^{-\sigma}$.
\end{proof}

\begin{Lemma}
\label{lem:org-reflection}
$\orgrefl S_{a,b}^\sigma = S_{a,b}^{-\sigma}$ and $\orgrefl C_{a,b}^\sigma = C_{a,b}^{-\sigma}$.

Thus $(|\col_n(S_{a,b}^\sigma)|)_{n}=(|\col_{-n}(S_{a,b}^{-\sigma})|)_{n}$ and $(|\col_n(C_{a,b}^\sigma)|)_{n}=(|\col_{-n}(C_{a,b}^{-\sigma})|)_{n}$.
\end{Lemma}

This means that reflection at the origin maps a staircase below the line to the staircase above the line and vice versa. This operation reverses the Beatty sequence of the staircase. See Figure~\ref{fig:org-reflection}.

\begin{proof} We compute
\begin{eqnarray*}
 (x,y)\in H_{a,b}^\sigma &\lrAr& \sigma y\leq \sigma(\frac{b}{a}x) \quad\lrAr\quad -\sigma(-y) \leq -\sigma (\frac{b}{a}(-x))\\
&\lrAr& (-x,-y)\in H_{a,b}^{-\sigma} \quad\lrAr\quad \orgrefl(x,y)\in H_{a,b}^{-\sigma}.
\end{eqnarray*}
which implies both $\orgrefl S_{a,b}^\sigma  = S_{a,b}^{-\sigma}$ and  $\orgrefl C_{a,b}^\sigma = C_{a,b}^{-\sigma}$, like in the proof of Lemma~\ref{lem:diag-reflection}.

For the second claim of the lemma we observe (using what we have already shown) that
\begin{eqnarray*}
 (n,y)\in \left(\col_n(S_{a,b}^\sigma)\right) &\lrAr& (n,y)\in  S_{a,b}^\sigma \quad\lrAr\quad (-n,-y)\in\; \orgrefl S_{a,b}^\sigma\\
&\lrAr& (-n,-y)\in S_{a,b}^{-\sigma} \quad\lrAr\quad (-n,-y)\in \left(\col_{-n}(S_{a,b}^{-\sigma})\right).
\end{eqnarray*}
As this gives us for any fixed $n$ a bijection between the sets $\col_n(S_{a,b}^\sigma)$ and ${\col_{-n}(S_{a,b}^{-\sigma})}$, their cardinality must be the same. The argument for $C_{a,b}^\sigma$ is analogous.
\end{proof}

Putting Lemmas~\ref{lem:translation2} and \ref{lem:org-reflection} together, we immediately obtain the non-obvious statement that reversing a Beatty sequence yields the same sequence up to shift.

\begin{Cor}
\label{cor:reversal}
$(|\col_n(S_{a,b}^\sigma)|)_n\equiv(|\col_{-n}(S_{a,b}^\sigma)|)_n$.
\end{Cor}

\begin{proof}
$ (|\col_n(S_{a,b}^\sigma)|)_n\equiv (|\col_n(S_{a,b}^{-\sigma})+v|)_n \equiv (|\col_n(S_{a,b}^{-\sigma})|)_n = (|\col_{-n}(S_{a,b}^{\sigma})|)_n$
\end{proof}

Similarly, Lemma~\ref{lem:translation1} implies that $C_{a,b}$ and $C_{a,b,r}$ have the same column sequence for any $r$.

\paragraph{Recursive Description of Staircases.} We now return to the operation called reduction, which we defined for balanced sequences in Section~\ref{sec:theproblem}. First, let us observe the relation between Beatty and Sturmian sequences more closely. The following fundamental lemma tells us that, not surprisingly, Sturmian sequences are Beatty sequences with $a>b$ and vice versa.

\begin{Lemma}
\label{Lem:reduction}
$\red{B}_{a,b}=B_{a,b \mod a}$.
Conversely if $s$ is a sequence balanced at $k\in\NN$ and $\red{s}=B_{a,b}$, then $s=B_{a,ak+b}$.
\end{Lemma}

\begin{proof}
By (\ref{eqn:Beatty-2}) we observe that for any $k\in\ZZ$ such that both $b$ and $b+ka$ are positive
\[
B_{a,b}(n) + k = B_{a,b+ka}(n).
\]
$B_{a,b}$ is balanced at $\floor{\frac{b}{a}}=b \div a$ by Lemma~\ref{Lem:beatty-is-balanced}. 
Note that by definition ${b \mod a} = {b - (b \div a)\,a}$.
So 
\[
\red{B}_{a,b}(n) = B_{a,b}(n) - b \div a = B_{a,b \mod a}(n).
\]
Conversely if $\red{s}(n)= B_{a,b}(n)$ and $s$ is balanced at $k$, then 
\[
s(n) = \red{s}(n) + k = B_{a,b}(n) + k = B_{a,b+ka}(n)
\]
\end{proof}

How can this relation be phrased in terms of the staircases $S_{a,b}$ and $S_{a,b \mod a}$? The following lemmas give an answer to this question. See $S_{5,13}$ and $S_{5,3}$ in Figure~\ref{fig:recursion}.

\begin{Lemma} 
\label{lem:reduction-bijection}
Let $0<a<b$. The lattice transformation $A=\begin{pmatrix}1 & 0 \\ b\div a & 1\end{pmatrix}$ gives a bijection between $C_{a,b}$ and $S_{a,b\mod a}$.
\end{Lemma}

The corners $C_{a,b}$ of the staircase $S_{a,b}$ are just the points of the smaller staircase $S_{a,b \mod a}$ up to a lattice transform. Here ``smaller'' refers to both the number of lattice points in a given interval and the encoding length of the two parameters $a$ and $b$. Note that the inverse of $A$ is $A^{-1}=\left(\begin{smallmatrix}1 & 0 \\ -b\div a & 1\end{smallmatrix}\right)$.

\begin{proof}
As $S_{a,b}$ is steep, $\col_n(C_{a,b})= \{(n,\floor{\frac{b}{a}n})\}$ and $\col_n(S_{a,b\mod a})= \{(n,\floor{\frac{b \mod a}{a}n})\}$ by Facts~\ref{fact:pipes} and~\ref{fact:toppoint}. But 
$$\begin{pmatrix} n \\ \floor{\frac{b}{a}n}\end{pmatrix} = \begin{pmatrix} n \\  (b \div a)n + \floor{\frac{b\mod a}{a}n}\end{pmatrix} = A\begin{pmatrix} n \\  \floor{\frac{b\mod a}{a}n}\end{pmatrix}.$$
\end{proof}

However, to obtain all points in $S_{a,b}$ from the corners $C_{a,b}$ we need to know which columns of $S_{a,b}$ are long and which are short (in the case $b>a$). It turns out that the corners in long columns are precisely the corners $C_{a,b\mod a}$ of the smaller staircase, again up to the lattice transformation $A$.

\begin{Lemma} 
\label{lem:reduction-long-columns}
Let $0<a<b$. Then $\col_n(C_{a,b\mod a})$ contains a point if and only if $\col_n(S_{a,b})$ is long.
\end{Lemma}

\begin{proof}
$S_{a,b\mod a}$ is flat. So we know by Fact~\ref{fact:interpretation} that $(n,\floor{\frac{b\mod a}{a}n})\in C_{a,b\mod a}$ if and only if $1=B_{a,b\mod a}(n)=\red{B}_{a,b}(n)$. But this just means that $\col_n(S_{a,b})$ is long.
\end{proof}

Taking the two lemmas together, we can describe every staircase $S_{a,b}$ in terms of the smaller corners and points in the smaller staircase $S_{a,b\mod a}$. This result, together with our ability to swap the parameters $a$ and $b$ (by means of Lemmas~\ref{lem:diag-reflection} and~\ref{lem:org-reflection}) and the fact that the staircases $S_{a,1}$ are easy to describe, we obtain a recursive characterization of all staircases. 

\vspace{1em}
Let us look at an example, which is shown in Figure~\ref{fig:recursion}, before we formulate the recursion formally in Lemma~\ref{lem:recursion}. We want to express $S_{5,13}$ in terms of smaller staircases. 

We know that the topmost points in each column are the points in $C_{5,13}$ and $C_{5,13}$ is just the image of $S_{5,3}$ under the lattice transformation $A=\left(\begin{smallmatrix}1 & 0 \\ 2 & 1\end{smallmatrix}\right)$. Note that $A$ keeps columns invariant.

We also know that $S_{5,13}$ has columns of lengths $2$ and $3$ and that the long columns are precisely those in which $C_{5,3}$ contains a point. So if we have an expression for $S_{5,3}$ and $C_{5,3}$, we can give an expression for $S_{5,13}$ and $C_{5,13}$.

To continue this argument inductively, we need to swap the parameters $a$ and $b$, but this we can achieve by reflecting the staircases at the origin and at the main diagonal. So we reduce the problem of describing $S_{5,3}$ to the problem of describing $S_{3,5}$. We can now continue in this fashion, expressing $S_{3,5}$ in terms of $S_{3,2}$, in terms of $S_{2,3}$, in terms of $S_{2,1}$.

At this point we have finally reached a staircase with integral slope. These staircases have the nice property that all columns and all rows are identical and hence they can be described by a simple expression: the Minkowski sum of the lattice points on a line with those in an interval. This entire process is illustrated in Figure~\ref{fig:recursion}.

\begin{figure}[ht]
\begin{center}
\input{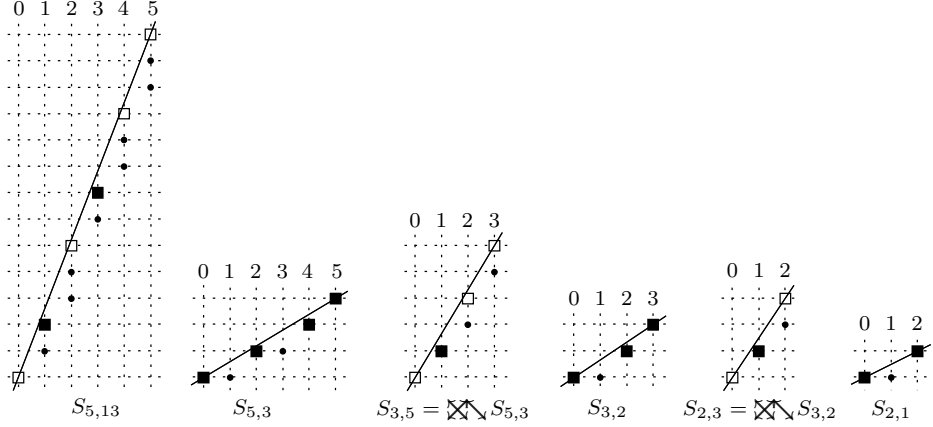}
\caption[]{\label{fig:recursion}This figure show the recursive process of expressing $S_{5,13}$ in terms of smaller staircases, described in the text. In this figure, empty squares indicate the corners of long columns, filled squares corners of short columns. Note that the empty squares occur in $S_{a,b}$ precisely in the columns, in which there is an element of $C_{a,b\mod a}$.}
\end{center}
\end{figure}

\begin{Lemma}
\label{lem:recursion}
Let $0<a,b\in\NN$ and $A=\begin{pmatrix}1 & 0 \\ b\div a & 1\end{pmatrix}$.
\begin{enumerate}
\item If $a<b$ and $\gcd(a,b)=1$, then
  \begin{eqnarray*}
    C_{a,b} & = & A S_{a,b\mod a} \\
    S_{a,b} & = & A S_{a,b\mod a} + \{ \twovec{0}{0} ,\ldots, \twovec{0}{-(b\div a) +1}\}\\
           &&\cup\; A C_{a,b\mod a} + \{ \twovec{0}{-(b\div a)}\}.
  \end{eqnarray*}
\item $C_{a,b} = \;\orgrefl\,\diagrefl C_{b,a}$ and $S_{a,b}  = \; \orgrefl\,\diagrefl S_{b,a} $
\item If $b=1$, then
  \begin{eqnarray*}
    C_{a,b} & = & \{\twovec{ka}{k}:k\in\ZZ\} \\
    S_{a,b} & = & \{\twovec{ka}{k}:k\in\ZZ\} + \{ \twovec{0}{0},\ldots,\twovec{a-1}{0} \}
  \end{eqnarray*}
\end{enumerate}
\end{Lemma}

In Section~\ref{sec:characterization} we use this recursive structure to develop a characterization of Sturmian sequences. In Section~\ref{sec:short-representations} we employ the recursion to obtain short rational functions that enumerate the lattice points inside lattice polytopes in the plane, and in Section~\ref{sec:dedekind-carlitz} for a representation of Dedekind-Carlitz polynomials that is computable in polynomial time.

\begin{proof}
\emph{1.} By Lemma~\ref{lem:reduction-bijection} the first equation holds. Every column of $S_{a,b}$ contains a corner and every column contains at least $(b\div a)$ points. So $A S_{a,b\mod a} + \{ \twovec{0}{0} ,\ldots, \twovec{0}{-(b\div a) +1}\}$ contains all points in $S_{a,b}$ except the bottom-most points of the long columns. By Lemma~\ref{lem:reduction-long-columns} the long columns are precisely those in which $S_{a,b\mod a}$ has a corner. So $A C_{a,b\mod a} + \{ \twovec{0}{-b\div a}\}$ is precisely the set of bottom-most points of the long columns of $S_{a,b}$.

\vspace{0.5em}
\emph{2.} $S_{a,b} = \;\orgrefl S^-_{a,b}=\;\orgrefl\,\diagrefl S_{b,a}$ by Lemmas~\ref{lem:diag-reflection} and~\ref{lem:org-reflection} and similarly for $C_{a,b}$.

\vspace{0.5em}
\emph{3.} If $b=1$, then for all $k,n\in\ZZ$ we have $\floor{\frac{b}{a}n}=k$ if and only if $ka\leq n \leq {(k+1)a-1}$. Hence $\row_k(S_{a,b})=\{ \twovec{ka}{k},\ldots,\twovec{ka+a-1}{k} \}$ and $\row_k(C_{a,b})=\{ \twovec{ka}{k} \}$.
\end{proof}

\paragraph{Relation to the Euclidean Algorithm.} This recursion is closely related to the Euclidean Algorithm, which takes as input two natural numbers $c_1,c_2\in\NN$. In each step $c_{i+1}={c_{i-1} \mod c_i}$ is computed. This continues until we reach a $j$ such that $c_{j+1}=0$ and $c_j\not=0$. Then $c_j=\gcd(c_1,c_2)$.

Now suppose we want to determine $S_{b,a}$ and $C_{b,a}$ for some $b>a$. We flip the two parameters and then reduce the staircase, i.e.\ we apply~\ref{lem:recursion}.2 and~\ref{lem:recursion}.1. This reduces the problem to computing $S_{a,b\mod a}$ and $C_{a,b\mod a}$. Again we flip and reduce, which reduces the problem to computing $S_{b\mod a, a \mod (b\mod a)}$ and $C_{b\mod a, a \mod (b \mod a)}$ and we continue in this fashion. In other words, we put $c_1=b$, $c_2=a$ and $c_{i+1}=c_{i-1} \mod c_i$ and compute the staircases $S_{c_i,c_{i+1}}$ and $C_{c_i,c_{i+1}}$ recursively, until we arrive at the case $S_{c_{j-1},1}$ and $C_{c_{j-1},1}$ which we can solve directly by~\ref{lem:recursion}.3. That we arrive in this case eventually follows by the correctness of the Euclidean Algorithm and the assumption that $\gcd(a,b)=1$! Note also that this recursion terminates after few iterations. This is made precise in the following lemma.

\begin{Lemma}
\label{lem:euclidean-algorithm}
Let $a,b\in\NN$ and let $(c_n)_{n\in\NN}$ denote the sequence defined by $c_1=b$, $c_2=a$ and $c_{i+2}=c_i \mod c_{i+1}$. Then $\min \set{j\in\NN}{c_{j+1}=0}\in\mathcal{O}(\log a)$.
\end{Lemma}

\begin{proof}
$(c_i)_{i\geq 2}$ is monotonously decreasing for all $a,b\in\NN$, as by definition $c_{i+2}=c_i \mod  c_{i+1} < c_{i+1}$. Thus $c_i \div c_{i+1}\geq 1$ for $i\geq 2$ and so
\[
c_i = (\underbrace{c_i \div c_{i+1}}_{\geq 1})c_{i+1} + (\underbrace{c_i \mod c_{i+1}}_{=c_{i+2}}) \geq \underbrace{c_{i+1}}_{\geq c_{i+2}} + c_{i+2} \geq  2c_{i+2}
\]
for $i\geq 2$. Hence $c_{i+2k}\leq 2^{-k}c_i$, and so if $k\geq \log_2 c_i$, then ${c_{i+2k}\leq 1}$. In particular the minimal $j$ such that $c_{j+1}=0$ satisfies $j\leq 2\log_2 c_2 +2\in\mathcal{O}(\log a)$.
\end{proof}

\paragraph{Recursive Description of Parallelepipeds.} Instead of describing the infinite set of lattice points in an entire staircase, one might want to describe finite subsets thereof, for example the set of lattice points in only ``one period'' of the staircase. We now give a recursion for the set of lattice points in the fundamental parallelepipeds of the cones $\cone\left(\twovec{a}{b},\twovec{1}{0}\right)$ and $\cone\left(\twovec{a}{b},\twovec{0}{-1}\right)$.

The \defn{cone} generated by $v_1,\ldots,v_n\in\RR^d$ is the set 
\[
\cone(v_1,\ldots,v_n)=\set{\sum_{i=1}^n \alpha_i v_i}{0\leq \alpha_i\in\RR\text{ for all }1\leq i\leq n}.
 \]
A cone is rational if all the $v_i$ are rational and it is simplicial if the $v_i$ are linearly independent. The \defn{fundamental parallelepiped} $\Pi_{\cone(v_1,\ldots,v_n)}$ of a simplicial cone $\cone(v_1,\ldots,v_n)$ is defined as
\[
 \Pi_{\cone(v_1,\ldots,v_n)}\define \set{\sum_{i=1}^n \alpha_i v_i}{0\leq \alpha_i < 1\text{ for all }1\leq i\leq n}.
\]
Note that any rational cone $\cone(v_1,\ldots,v_n)\subseteq \RR^m$ can be transformed unimodularly to a rational cone $\cone(\sigma e_j,v'_1,\ldots,v'_n)$ with $\sigma\in\{+,-\}$ and $1\le j \le m$. So we don't restrict ourselves by only looking at cones containing $e_1$ or $-e_2$ in the generators.

With the above notation 
\begin{eqnarray*}
 \vpipe{a}{b}{}\cap [0,a) \times \RR &=& \Pi_{\cone\left(\twovec{0}{-1},\twovec{a}{b}\right)}\cap\ZZ^2\\
 \hpipe{a}{b}{}\cap \RR \times [0,b) &=& \Pi_{\cone\left(\twovec{1}{0},\twovec{a}{b}\right)}\cap\ZZ^2,
\end{eqnarray*}
see Figure \ref{fig:staircase-and-parallelepiped}. This means that if $a<b$ (and hence $S_{a,b}=\hpipe{a}{b}{}$), the points $z$ in the staircase $S_{a,b}$ with ${0\leq z_2 < b}$ are just the lattice points in the fundamental parallelepiped $\Pi_{\cone\left(\twovec{1}{0},\twovec{a}{b}\right)}$. The corners $z\in C_{a,b}$ with $0 \leq z_1 < a$ are just the lattice points in the fundamental parallelepiped $\Pi_{\cone\left(\twovec{0}{-1},\twovec{a}{b}\right)}$. 

\begin{figure}[ht]
\begin{center}
\input{parallelepiped_scaled.pstex_t}
\caption[]{\label{fig:staircase-and-parallelepiped}If we intersect $\vpipe{a}{b}{}$ with $[0,a)\times \RR$ we obtain the fundamental parallelepiped of the cone generated by $\twovec{0}{-1}$ and $\twovec{a}{b}$.}
\end{center}
\end{figure}

To give an interpretation of our recursion in terms of fundamental parallelepipeds it is convenient to define the set $\ppo_{\cone(v_1,\ldots,v_n)}$ of lattice points (!)~in the open fundamental parallelepiped of ${\cone(v_1,\ldots,v_n)}$ as
\[
 \ppo_{\cone(v_1,\ldots,v_n)}\define \ZZ^2 \cap \set{\sum_{i=1}^n \alpha_i v_i}{0< \alpha_i < 1}.
\]
Note that if $n=2$ and both $v_1$ and $v_2$ are primitive, then $\ppo_{\cone(v_1,v_2)}\cup \{\twovec{0}{0}\}= \ZZ^2 \cap\Pi_{\cone(v_1,v_2)}$. So it suffices to give a recursion for the sets of lattice points in open fundamental parallelepipeds.

We are going to use the following abbreviations:
\begin{eqnarray*}
\Pi_{\downarrow,a,b} := \Pi_{\cone(\twovec{0}{-1},\twovec{a}{b})} && \ppo_{\downarrow,a,b} := \ppo_{\cone(\twovec{0}{-1},\twovec{a}{b})}\\
\Pi_{\rightarrow,a,b} := \Pi_{\cone(\twovec{1}{0},\twovec{a}{b})} && \ppo_{\rightarrow,a,b} := \ppo_{\cone(\twovec{1}{0},\twovec{a}{b})}
\end{eqnarray*}

 In terms of open parallelepipeds, Lemma~\ref{lem:recursion} can now be phrased as follows. An example illustrating the somewhat involved expression in \ref{lem:recursion-parallelepiped}.1 is given in Figure~\ref{fig:parallelepiped-recursion}.

\begin{figure}[ht]
\begin{center}
\input{ppd_recursion_scaled.pstex_t}
\caption[]{\label{fig:parallelepiped-recursion}This figure illustrates the formula given in \ref{lem:recursion-parallelepiped}.1. $\pC{5}{13}$ is expressed in terms on $\pS{5}{3}$ (shown) and $\pC{5}{3}$ (the corners of $\pS{5}{3}$). The idea is the same as in Figure~\ref{fig:recursion} and Lemma~\ref{lem:recursion}. However there is one important difference: Both $\col_5(\pS{5}{3})$ and $\col_5(\pC{5}{3})$ are empty. These have to be added using the third term $\twovec{5}{13} + \{ \twovec{0}{-1}, \twovec{0}{-2 }\}$. }
\end{center}
\end{figure}

\begin{Lemma}
\label{lem:recursion-parallelepiped}
Let $a,b\in\NN$ and $A=\begin{pmatrix}1 & 0 \\ b\div a & 1\end{pmatrix}$.
\begin{enumerate}
\item If $0<a<b$ and $\gcd(a,b)=1$, then
  \begin{eqnarray*}
    \pS{a}{b} & = & A \pS{a}{b\mod a} \\
    \pC{a}{b} & = & A \pS{a}{b\mod a} + \{ \twovec{0}{0} ,\ldots, \twovec{0}{-(b\div a) +1}\} \\ 
                        && \cup \; A \pC{a}{b\mod a} + \twovec{0}{-b\div a}\\
                        && \cup \; \twovec{a}{b} + \{ \twovec{0}{-1} ,\ldots, \twovec{0}{-b\div a }\}.
  \end{eqnarray*}
\item $\pC{a}{b} = \orgrefl\,\diagrefl \pS{b}{a} +  \twovec{a}{b}$ and $\pS{a}{b} =  \orgrefl\,\diagrefl \pC{b}{a} + \twovec{a}{b}$.
\item $\pC{a}{1}  =  \emptyset$ and $\pS{a}{1}  =  \{ \twovec{1}{0},\ldots,\twovec{a-1}{0} \}$.
\end{enumerate}
\end{Lemma}

This allows us describe $\pS{a}{b}$ in terms of $\pS{a}{b\mod a}$ and $\pC{a}{b\mod a}$. The proof is similar to the one of Lemma~\ref{lem:recursion} and we omit it for brevity.

\vspace{1em}

\paragraph{Recursive Description of Triangles.} We conclude this section by giving a similar recursion for triangles, see Figure~\ref{fig:triangle-recursion}. We write 
\[
\Delta_{a,b}:=\conv\left\{\twovec{0}{0},\twovec{a}{0},\twovec{a}{b}\right\}
\]
and 
\[
\Delta'_{a,b}:=\Delta_{a,b}\setminus\conv\left\{\twovec{0}{0},\twovec{a}{b}\right\}
\]
 to denote closed and half-open triangles, respectively. The corresponding lattice point sets are denoted by $T_{a,b}:=\Delta_{a,b}\cap \ZZ^2$ and $T'_{a,b}:=\Delta'_{a,b}\cap \ZZ^2$.

\vspace{1em}
The idea is now that for $0<a<b$ the triangle $\Delta_{a,b}$ can be decomposed into two parts $\Delta'_{a,(b\div a)a}$ and $A\Delta_{a,b\mod a}$. The former is defined by a line with integral slope and hence the set of lattice points $T'_{a,(b\div a)a}$ is easy to describe. The latter can be transformed into $\Delta_{b\mod a,a}$ and we can obtain a description of the lattice point set $T_{b\mod a,a}$ recursively. See Figure~\ref{fig:triangle-recursion}. The resulting recursion is given in Lemma~\ref{lem:recursion-triangle} without proof.

\begin{figure}[ht]
\begin{center}
\input{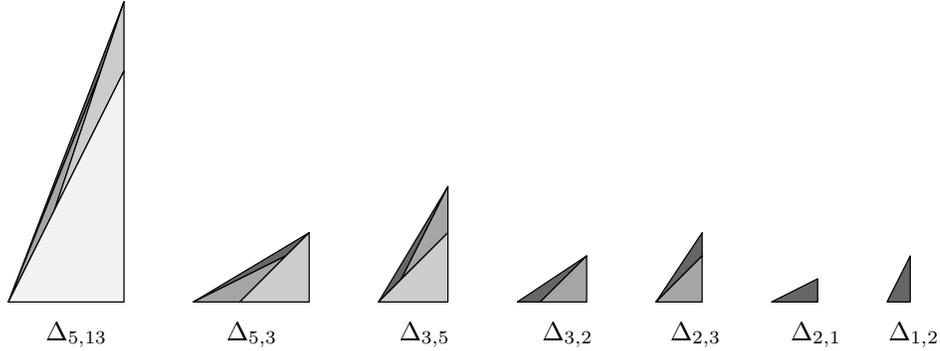}
\caption{\label{fig:triangle-recursion}Similarly to our recursive description of $S_{5,13}$ (see Figure~\ref{fig:recursion}), we can apply Lemma~\ref{lem:recursion-triangle} recursively to partition $\Delta_{5,13}$ into triangles with integral slope. The different shadings indicate which triangle the different regions correspond to.}
\end{center}
\end{figure}

\begin{Lemma}
\label{lem:recursion-triangle}
Let $a,b\in\NN$ and $A=\begin{pmatrix}1 & 0 \\ b\div a & 1\end{pmatrix}$.
\begin{enumerate}
\item $T_{a,b} = A T_{a,b\mod a} \cup T'_{a,(b\div a)a}$.
\item $T_{a,b}=\orgrefl\,\diagrefl T_{b,a} + \twovec{a}{b}$.
\item $T_{a,1}=\set{\twovec{k}{0}}{0\leq k\leq a} \cup \{ \twovec{a}{1}\}$.
\item If $k\in\NN$, then 
\begin{eqnarray*}
T'_{a,ka} &=& \bigcup_{0<l\leq a} \left\{ \twovec{l}{0},\ldots,\twovec{l}{lk-1} \right\} \\
&= & \set{\twovec{0}{m}}{m\in\NN} + \{\twovec{0}{0},\twovec{1}{0},\ldots,\twovec{a}{0}\}\\
  &&   \setminus \left(\set{\twovec{0}{m}}{m\in\NN} + \{\twovec{0}{0},\twovec{1}{k},\ldots,\twovec{a}{ak}\}\right).
\end{eqnarray*}
\end{enumerate}
\end{Lemma}

The advantage of using the second expression for $T'_{a,ka}$ in~\ref{lem:recursion-triangle}.4 will become clear in Section~\ref{sec:short-representations} where we use it to obtain a short rational function representing the generating function of the set of lattice points inside $T'_{a,ka}$. Note that we obtain a recursion formula for $T'$ by replacing every occurrence of $T$ in~\ref{lem:recursion-triangle}.1 and~\ref{lem:recursion-triangle}.2 with $T'$ and replacing~\ref{lem:recursion-triangle}.3 with $T'_{a,1}=\set{\twovec{k}{0}}{1\leq k\leq a}$.

In \cite{KNA94} Kanamaru et al.\ use a recursive procedure as in Lemma~\ref{lem:recursion-triangle} to give an algorithm to enumerate the set of lattice points \emph{on} a line segment. They go on to give an algorithm that enumerates lattice points inside triangles using the transformation $A$, however in this case they do not apply recursion and do not mention the partition given in Lemma~\ref{lem:recursion-triangle}.1. This partition however is observed by Balza-Gomez et al.\ in \cite{BMM99}. But as they are interested in giving an algorithm for computing the convex hull of lattice points strictly below a line segment, they do not work with the full set of lattice points $T_{a,b}$. In both cases no explicit recursion formula such as Lemma~\ref{lem:recursion-triangle} is given.

%%% Local Variables: 
%%% mode: latex
%%% TeX-master: "Staircases_in_Z2"
%%% End: 

% \input{characterizations}
%%%%%%%%%%%%%%%%%%%%%%%%%%%%%%%%%%%%%
%                                   %
%             section 4             %
%                                   %
%%%%%%%%%%%%%%%%%%%%%%%%%%%%%%%%%%%%%

\section{Characterizations of Sturmian Sequences}
\label{sec:characterization}

In this section we state several characterizations of Sturmian sequences of rational numbers, i.e.\ sequences of the form $\red{B}_{a,b}$ (or equivalently $B_{a,b}$ with $0<b\le a$). We will first motivate each characterization in a separate paragraph without proofs and then summarize them in Theorem~\ref{thm:equivalences}. The proof of the theorem occupies Section~\ref{sec:theproof}.

\paragraph*{Recursive Structure.} The most important characterization of Sturmian sequences for our purposes is a recursive one. It is based on the concept of reduction presented in Section~\ref{sec:geometricidea} that relates $\red{B}_{a,b}$ to $\red{B}_{a,b \mod a}$ in a way reminiscent of the Euclidean Algorithm.

We first present the idea informally. Let $0<a<b$ and consider the sequence $B_{b,a}$ and the related staircase $S_{b,a}$. An interval of $B_{b,a}$ of the form $10\ldots0$ of length $k$ corresponds to a corner $c\in S_{b,a}$ and $k-1$ points in $S_{b,a}$ at the same height as $c$. We call a maximal interval of the form $10\ldots0$ a block. A block of $B_{b,a}$ corresponds to a row of $S_{b,a}$. If the block has length $k$, the row contains $k$ points. The \defn{block sequence} $m(B_{b,a})$ of $B_{b,a}$ is the sequence of block lengths of $B_{b,a}$. By the above observation the block sequence of $B_{b,a}$ is the row sequence of $S_{b,a}$ and by Lemma~\ref{lem:recursion}.2 and Corollary~\ref{cor:reversal} the column sequence of $\;\orgrefl\diagrefl S_{b,a}=S_{a,b}$ \emph{up to shift}. This means that $m(B_{b,a})\equiv B_{a,b}$ is Beatty and hence $\red{m}(B_{b,a})\equiv B_{a,b\mod a}$ is Sturmian. See Figure~\ref{fig:example-recursive} for an example. It turns out that this gives a recursive characterization of Sturmian sequences: a sequence $s$ is Sturmian iff $m(s)$ is balanced and $\red{m}(s)$ is Sturmian.

\begin{figure}[hbt]
\begin{center}
 \includegraphics[scale=0.3]{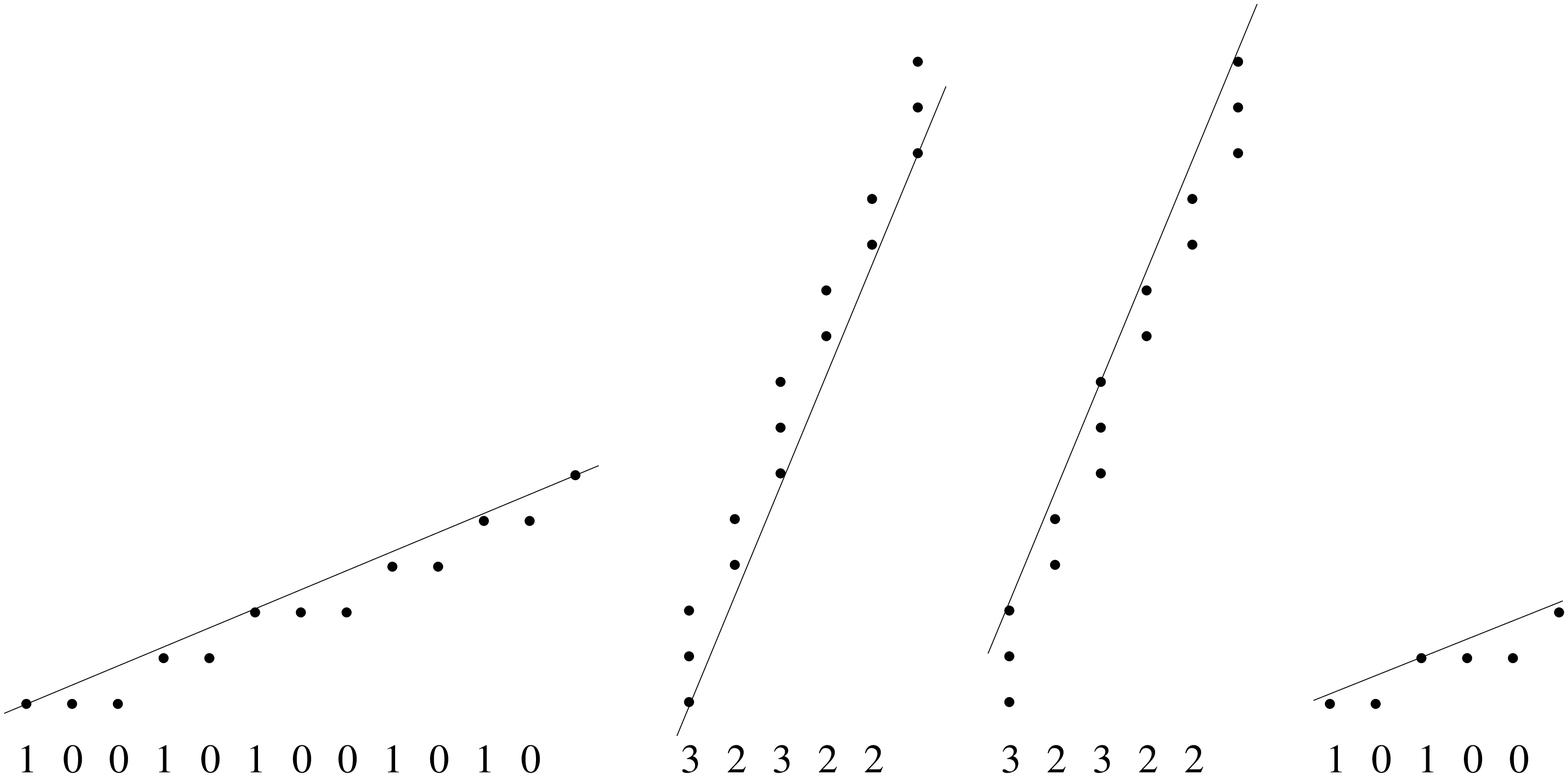}
\end{center}
\caption{\emph{Recursive structure of Sturmian sequences.} This figure shows several staircases $S_{a,b}$ and the corresponding parts of the associated Beatty sequence $B_{a,b}$. In each picture the first and last column correspond to the same element of $B_{a,b}$ modulo the minimal period. The first picture shows $S_{12,5}$ and $B_{12,5}$. $B_{12,5}$ records which columns contain corners, so the block sequence $m(B_{12,5})$ is just the row sequence of $S_{12,5}$. Reflecting at the main diagonal turns rows into columns, so $m(B_{12,5})$ is the column sequence of $\diagrefl S_{12,5}= S^-_{5,12}$, which is shown in the second picture. Note that in the first picture the part of the staircase that we show was chosen such that points from the first and the last column lie on the defining line - and this property is preserved under reflection at the main diagonal. Applying the reflection at the origin gets us to $\orgrefl\diagrefl S_{12,5}=S_{5,12}$, which is shown in the third picture along with its column sequence $B_{5,12}$. The reflection at the origin, however, reverses the column sequence. Fortunately Corollary~\ref{cor:reversal} tells us that Sturmian sequences are invariant under reversals - up to a shift. This shift in the column sequence can be seen from the fact that in the third picture, the defining line does not pass through points of the first or last column anymore. From the first three pictures, we see that the column sequence of $S_{5,12}$ is just the row sequence of $S_{12,5}$ up to shift, i.e.\ $m(B_{12,5})\equiv B_{5,12}$. Now we can apply reduction. We pass from $S_{5,12}$ to $S_{5,2}$, as shown in the last picture, which gives us $\red{m}(B_{12,5})\equiv B_{5,2}$. This is a geometric illustration of the combinatorial fact that if $s$ is Sturmian, so is $\red{m}(s)$.}
\label{fig:example-recursive}
\end{figure}

\vspace{0.5em}

To make this precise we give the following definitions. Suppose $s=(s_n)_{n\in\ZZ}$ is a periodic 0,1-sequence with $\PP(s)=a$ and $\ones(s)=b$. Without loss of generality we can assume that $s_{0}=1$. Let $i_1,\ldots,i_b\in\{0,\ldots,a-1\}$ be the indices of the 1s, i.e.\ let $i_1<i_2<\ldots<i_b$ with $s_{i_j}=1$ for all $1\leq j \leq b$. Put $i_{b+1}:=a$. Then the $j$-th block in $(s_0,\ldots,s_{a-1})$ is $\interval{s}{i_{j}}{i_{j+1}-1}$ and the length of the $j$-th block is $m_j:=i_{j+1} - i_{j}$. The \defn{block sequence} $m(s)$ is the infinite periodic sequence generated by $(m_1,\ldots,m_b)$. Note that this definition determines $m(s)$ only up to shift, which suffices for our purposes. On equivalence classes of sequences up to shift, $m$ is an injective function, i.e.\ $s_1\equiv s_2 \lrAr m(s_1)\equiv m(s_2)$. We call a sequence \defn{block balanced} if it is balanced and its block sequence is balanced. In this case we can consider the reduced block sequence $\red{m}(s)$ which is again a $0,1$-sequence. A sequence $s$ is \defn{recursively balanced} 
\begin{itemize}
\item if $\ones(s)=1$, or
\item if $s$ is block balanced and $\red{m}(s)$ is recursively balanced.
\end{itemize}
The characterization now is this:

\begin{center}
 {\it{A periodic 0,1-sequence is Sturmian if and only if it is recursively balanced.}}
\end{center}

\paragraph*{Even Distribution of 0s and 1s.} Common sense suggests that, as the staircase approximates a line, the 0s and 1s of the Sturmian sequence should be distributed as evenly as possible. The actual number of 1s in every interval should be as close as possible to the expected number of 1s. This can be made precise in the following way. On an interval of length $l$, a line with slope $\frac{b}{a}$ increases by $\frac{b}{a}l$. So the expected number of 1s in an interval of length $l$ of $\red{B}_{a,b}$ is $\frac{b}{a}l$, if $b<a$, and $\frac{b \mod a}{a}l$ in general. As $\frac{b \mod a}{a}l$ is in general not an integer, the best that can be hoped for is that for every interval $I$ of length $l$ the number of 1s contained in $l$ is either $\floor{\frac{b \mod a}{a}l}$ or $\ceil{\frac{b \mod a}{a}l}$, and indeed this is a necessary and sufficient characterization of Sturmian sequences. Formally, we say that the 1s in a periodic 0,1-sequence $s$ are \defn{evenly distributed} if for every interval $\interval{s}{x_0}{x_1}$
\begin{eqnarray}
\label{eqn:evenly-def-1}
\txtones(\interval{s}{x_0}{x_1}) & \in & \left\{ \floor{\frac{\ones(s)}{\PP(s)}\txtlength(\interval{s}{x_0}{x_1})}, \ceil{\frac{\ones(s)}{\PP(s)}\txtlength(\interval{s}{x_0}{x_1})}   \right\}.
\end{eqnarray}
Note that if $z\in\ZZ$ and $r\in\RR$, then $z\in\left\{ \floor{r}, \ceil{r} \right\}$ if and only if $z-1 < r < z+1$, so the condition
\begin{eqnarray}
\label{eqn:evenly-def-2}
\txtones(\interval{s}{x_0}{x_1}) -1\quad < & \frac{\ones(s)}{\PP(s)}\txtlength(\interval{s}{x_0}{x_1}) & < \quad \txtones(\interval{s}{x_0}{x_1}) +1
\end{eqnarray}
is equivalent to (\ref{eqn:evenly-def-1}). If an interval $\interval{s}{x_0}{x_1}$ violates the left-hand inequality, then we say it contains too many 1s and if it violates the right-hand inequality, we say it contains too few 1s. The characterization, then, is this.

\begin{center}
 {\it{A periodic $0,1$-sequence $s$ is Sturmian if and only if the 1s in $s$ are evenly distributed.}}
\end{center}

This characterization appears in \cite{GLL78} and was later improved in \cite{Fraenkel05}.

\paragraph*{Symmetry.} A different way to phrase that the 0s and 1s are distributed evenly would be to state that Sturmian sequences are symmetric. If symmetric is taken to mean invariant under reversals (up to shift), then this is a true statement (Corollary~\ref{cor:reversal}) - but insufficient to characterize Sturmian sequences. 

However, Lemma~\ref{lem:translation} suggests a different notion of symmetry. % See Figure~\ref{fig:examples-symmetry}. 
If we start with a flat staircase $S_{a,b}$ with $0<b<a$ and move the defining line downwards by a small amount, the resulting staircase $S_{a,b,r}$ will be a translate of $S_{a,b}$. Hence their column sequences are identical up to shift. But using Lemma~\ref{lem:dist-to-line}, we see that if $0>r\ge-\frac{1}{a}$ the only columns that differ are $\col_{ka}$ for $k\in\ZZ$: the single point in these columns has been moved down by one. The columns $\col_{ka}$ do not contain a corner anymore, whereas the columns $\col_{ka+1}$ do. In the Sturmian sequence, this translates to taking the corresponding interval $1,0$ and replacing it with the interval~$0,1$. 

\vspace{1em}
This observation gives rise to the notion of swap symmetry. A periodic $0,1$-sequence $s$ is swap symmetric if there is a pair $(s_i,s_{i+1})=(1,0)$ such that if we replace this pair and all periodic copies of it by $(0,1)$, we obtain a sequence $s'$ that is identical to $s$ up to a shift. See Figure~\ref{fig:examples-symmetry} for an example. Formally, given a periodic sequence $s$ and $i\in\ZZ$, we define the sequence $\swap{s}{i}:=(\swap{s}{i}_n)_n$ by
$$\swap{s}{i}_n = 
\begin{cases} 
s_n - 1& \text{if $n\equiv i \mod \PP(s)$} \\
s_n + 1& \text{if $n\equiv i+1 \mod \PP(s)$} \\
s_n & \text{otherwise}
\end{cases}.
$$
We call a periodic sequence $s$ \defn{swap symmetric} if there exists an $i\in\ZZ$ such that $s\equiv \swap{s}{i}$. Note that if $s \equiv \swap{s}{i}$ then 
\begin{eqnarray*}
s_i &=& \swap{s}{i}_{i+1} \quad = \quad s_{i+1}+1\\
s_{i+1} &=& \swap{s}{i}_{i}\quad = \quad s_{i}-1
\end{eqnarray*}
as the number of entries $0\leq k<\PP(s)$ such that $s_k=c$ cannot change under a swap for any constant $c\in\NN$, for swap symmetric $s$.

\begin{figure}
\begin{center}
 \includegraphics[scale=0.5]{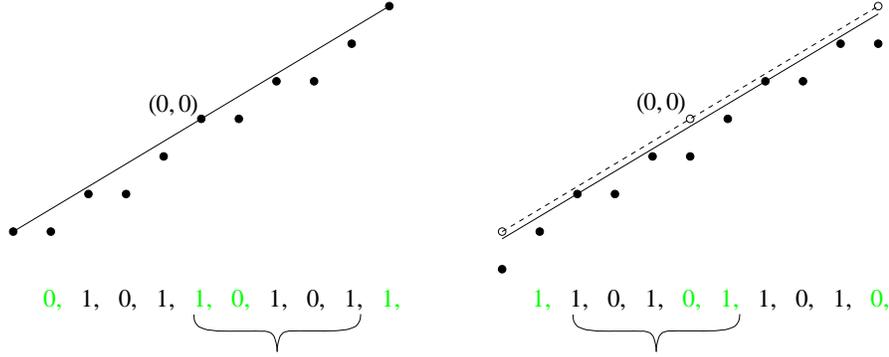}
\end{center}
\caption{\emph{Swap symmetry of Stumian sequences.} Consider the staircase $S_{5,3}$ and the corresponding Sturmian sequence $B_{5,3}$, which records the columns of $S_{5,3}$ that contain a corner. If we move the line defining $S_{5,3}$ downwards by a small amount, then only those columns change in which there was a point on the line $L_{5,3}$. These columns do not contain corners any more - the corners move one column to the right. So in the sequence $B_{5,3}$ this corresponds to replacing an interval 10 with 01, i.e.\ by swapping a 1 and a 0. The digits that are swapped are highlighted in the figure. Now, by Lemma~\ref{lem:translation} translating the defining line only shifts the column sequence of $C_{5,3}$. So the sequence obtained from $B_{5,3}$ by swapping is again $B_{5,3}$ up to shift. This is shown by the braces, which indicate minimal periods of both sequences.}
\label{fig:examples-symmetry}
\end{figure}

This property characterizes Sturmian sequences:

\begin{center}
 {\it{A periodic 0,1-sequence is Sturmian if and only if it is swap-symmetric.}}
\end{center}

\vspace{1em}
Having motivated the three characterizations, we can now state the theorem.

\begin{Thm}\label{thm:equivalences}
Let $s=(s_n)_{n\in\ZZ}$ be a periodic 0,1-sequence with  $\PP(s)=a$ and $\ones(s)=b\geq 1$. Then the following are equivalent:
\begin{enumerate}
\renewcommand{\theenumi}{{\it (\roman{enumi})}}
\renewcommand{\labelenumi}{(\roman{enumi})}
\item \label{eq:sturmian} $s\equiv\red{B}_{a,b}$.
\item \label{eq:rec-bal} $s$ is recursively balanced.
\item \label{eq:even} The 1s in $s$ are evenly distributed.
\item \label{eq:symmetry} $s$ is swap symmetric.
\end{enumerate}
\end{Thm}

{\ref{eq:sturmian}~$\rAr$~\ref{eq:even}} is easy to prove (see next section) and \ref{eq:even}~$\rAr$~\ref{eq:sturmian} appears in \cite{GLL78}, although the concept of ``nearly linear'' sequences used in \cite{GLL78} differs slightly from \ref{thm:equivalences}{\it(iii)}. The connection between these definitions\footnote{In \cite{Fraenkel05} sequences with an even distribution of 1s are called ``balanced''.} is made in \cite{Fraenkel05}, where the result from \cite{GLL78} is extended. In both cases the focus lies on the more general case of lines with irrational slope. The proofs given in these two sources differ from the proofs we present in Section~\ref{sec:theproof}. 

As far as we know the concepts of recursively balanced and swap symmetric sequences do not appear in the prior literature.

%%% Local Variables: 
%%% mode: latex
%%% TeX-master: "Staircases_in_Z2"
%%% End: 

% \input{proof}
%%%%%%%%%%%%%%%%%%%%%%%%%%%%%%%%%%%%%
%                                   %
%             section 5             %
%                                   %
%%%%%%%%%%%%%%%%%%%%%%%%%%%%%%%%%%%%%

\section{Proof of the Characterizations}
\label{sec:theproof}

We first show the recursive characterization,  i.e.\ that a periodic 0,1-sequence is Sturmian if and only if it is recursively balanced. To that end we first prove two lemmas.

\begin{Lemma}
\label{lem:block-sequences}
If $0<a<b$, then $m(B_{b,a})\equiv B_{a,b}$.
\end{Lemma}

\begin{proof}
Let $0<a<b$. A maximal interval of $B_{b,a}$ of the form $1,0,\ldots,0$ corresponds to a row of $S_{b,a}$, so 
$$m(B_{b,a})\equiv (|\row_n(S_{b,a})|)_n = (|\col_n(S^-_{a,b})|)_n \equiv (|\col_n(S_{a,b})|)_n = B_{a,b}$$
where we use Lemma~\ref{lem:diag-reflection} in the second and Lemma~\ref{lem:translation} in the third step.
\end{proof}

From this we also get that for $a>b$ (i.e.\ flat staircases) the sequence $m(B_{a,b})$ is balanced, and thus Sturmian sequences are block balanced.

\begin{Lemma}
\label{lem:block-sequence-recursion}
If $s$ is a block balanced $0,1$-sequence, then $s$ is Sturmian if and only if $\red{m}(s)$ is Sturmian.
\end{Lemma}

\begin{proof}
As $s$ is a block balanced 0,1-sequence
\begin{eqnarray*}
&& \text{$s\equiv B_{b,a}$ for some $0<a<b$} \\
&\lrAr& \text{$m(s)\equiv B_{a,b}$ for some $0<a<b$} \\
&\lrAr& \text{$\red{m}(s)\equiv \red{B}_{a,b}$ for some $0<a<b$} 
\end{eqnarray*}
where the first equivalence holds by Lemma~\ref{lem:block-sequences} and the fact that $m$ is injective %\footnote{$s_1\equiv s_2 \lrAr m(s_1)\equiv m(s_2)$} 
and the second equivalence holds by Lemma~\ref{Lem:reduction}.
\end{proof}

\begin{proof}[Proof of Theorem~\ref{thm:equivalences}:~\ref{eq:sturmian} $\lrAr$~\ref{eq:rec-bal}] The proof is by induction on $\ones(s)$. If $\ones(s)=1$, the statement holds. For the induction step, we have the following equivalences:
\begin{eqnarray*}
&& \text{$s$ is recursively balanced} \\
&\lrAr& \text{$s$ is block balanced and} \\
&& \text{$\red{m}(s)$ is recursively balanced} \\
&\lrAr& \text{$s$ is block balanced and} \\
&& \text{$\red{m}(s)$ is Sturmian} \\
&\lrAr& \text{$s$ is block balanced and} \\
&& \text{$s$ is Sturmian} \\
&\lrAr& \text{$s$ is Sturmian}
\end{eqnarray*}
Here we use Lemma~\ref{lem:block-sequence-recursion} in the third step. Note that the induction terminates in the case $\ones(s)=1$ since if $\ones(s)>1$, then $s$ has blocks of different sizes and so $\ones(s)>\ones(\red{m}(s))\geq 1$.
\end{proof}

Now we turn to the proof of the characterization that a periodic 0,1-sequence $s$ is Sturmian if and only if the 1s in $s$ are evenly distributed. One direction is easy to show.

\begin{proof}[Proof of Theorem~\ref{thm:equivalences}:~\ref{eq:sturmian} $\rAr$~\ref{eq:even}]
Let $B_{a,b}$ be a Sturmian sequence (i.e.\ $a>b$) and let $\interval{B_{a,b}}{x_0}{x_1}$ be any interval of $B_{a,b}$. 
Using (\ref{eqn:Beatty-2}) and the fact that 
$$\frac{b}{a}(x_1-x_0 +1) = \frac{\ones(B_{a,b})}{\PP(B_{a,b})}\txtlength(\interval{B_{a,b}}{x_0}{x_1})$$  
we obtain
\begin{eqnarray}
&& \txtones(\interval{B_{a,b}}{x_0}{x_1}) = \frac{\ones(B_{a,b})}{\PP(B_{a,b})}\txtlength(\interval{B_{a,b}}{x_0}{x_1}) + \left\{ \frac{b}{a}(x_0-1) \right\} - \left\{ \frac{b}{a}x_1 \right\}. \label{eqn:sturmian->even}
\end{eqnarray}
So, since $0\leq \left\{ \frac{b}{a}(x_0-1) \right\}<1$ and $0\leq \left\{ \frac{b}{a}x_1 \right\}< 1$ and both terms appear in (\ref{eqn:sturmian->even}) with opposite signs,
\begin{eqnarray*}
\txtones(\interval{B_{a,b}}{x_0}{x_1}) -1 & <  \frac{\ones(B_{a,b})}{\PP(B_{a,b})}\txtlength(\interval{B_{a,b}}{x_0}{x_1}) < & \txtones(\interval{B_{a,b}}{x_0}{x_1}) + 1. 
\end{eqnarray*}
\end{proof}

To show the other direction, we make use of the first characterization. We show that if the 1s in $s$ are evenly distributed, then $s$ is recursively balanced.

\begin{proof}[Proof of Theorem~\ref{thm:equivalences}:~\ref{eq:even} $\rAr$~\ref{eq:rec-bal}]
Let $s$ be a periodic 0,1-sequence in which the 1s are evenly distributed. We use induction on $\ones(s)$. If $\ones(s)=1$, then by definition $s$ is recursively balanced. For the induction step, we assume $\ones(s)>1$ and show that~$s$ is block balanced and the 1s in $\red{m}(s)$ are evenly distributed. Then we can apply the induction hypothesis to obtain that $\red{m}(s)$ and hence $s$ is recursively balanced.

\vspace{0.5em}
\emph{Step 1: $s$ is block balanced.} If there were blocks of zeros in $s$ that differed in length by at least two, then we could find intervals $u$ and $v$ of the same length $l$, such that~$u$ contains two 1s and $v$ contains none. But then $\{0,2\}\subseteq\{\floor{\frac{\ones(s)}{\PP(s)}l},\ceil{\frac{\ones(s)}{\PP(s)}l}\}$, which is impossible. So $s$ is block balanced at some $k\in\NN$, $\red{m}(s)$ is well defined and the following identities hold.
\begin{eqnarray*}
\PP(\red{m}(s)) &=& \ones(s) \\
\ones(\red{m}(s)) &=& \PP(s) - k\ones(s)
\end{eqnarray*}

\emph{Step 2: The 1s in $\red{m}(s)$ are evenly distributed.} Briefly, the idea is this: if $m'$ is an interval of $\red{m}(s)=:m$ that has too many 1s, looking at the corresponding interval in $s$ we will find many large blocks (i.e.\ many 0s) and so we can construct an interval~$s''$ of~$s$ that has too few 1s. This gives a contradiction to the assumption that the 1s in $s$ are evenly distributed.

Let $m'=\interval{\red{m}(s)}{x_0}{x_1}$ be an interval of $\red{m}(s)$.  We have to show that
\begin{eqnarray}
\label{eqn:even->rec-bal 1}
\txtones(m') - 1 < \frac{\ones(m)}{\PP(m)}\txtlength(m') < \txtones(m') + 1.
\end{eqnarray}
Assume to the contrary that $m'$ violates (\ref{eqn:even->rec-bal 1}).

We first argue that without loss of generality 
\begin{eqnarray}
\label{eqn:even->rec-bal 2}
\txtones(m')-1\geq \frac{\ones(m)}{\PP(m)}\txtlength(m'),
\end{eqnarray}
 i.e.\ that $m'$ contains too many 1s. Suppose $m'$ contains too few 1s, i.e.\ $\txtones(m')+1 \leq \frac{\ones(m)}{\PP(m)}\txtlength(m')$. Then choose $x_2>x_1$ such that the length of the interval $\interval{m}{x_0}{x_2}$ is a multiple $\alpha\PP(m)$, $\alpha\in\NN$ of the period length. Then $\txtones(\interval{m}{x_0}{x_2})=\alpha\ones(m)$ as~$m$ is periodic. Now the interval $\interval{m}{x_1+1}{x_2}$ has too many 1s, i.e.\ $\txtones(\interval{m}{x_1+1}{x_2})-1 \geq \frac{\ones(m)}{\PP(m)}\txtlength(\interval{m}{x_1+1}{x_2})$, as witnessed by the following computation:
\begin{eqnarray*}
\txtones(\interval{m}{x_1+1}{x_2}) &=&  \txtones(\interval{m}{x_0}{x_2})-\txtones(m') = \alpha\ones(m) - \txtones(m') \\
& \geq & \alpha\ones(m) -  \frac{\ones(m)}{\PP(m)}\txtlength(m') + 1 \\
& = & \frac{\ones(m)}{\PP(m)} ( \alpha\PP(m) - \txtlength(m')) + 1 \\
& = & \frac{\ones(m)}{\PP(m)} \txtlength(\interval{m}{x_1+1}{x_2}) + 1.
\end{eqnarray*} 

Each element of $m'$ corresponds to a block of 0s in $s$, where we take the block to include the preceding 1 but not the succeeding 1. Taking all the blocks in $s$ together that correspond to elements of $m'$ we obtain an interval $s'=\interval{s}{y_0}{y_1}$ of~$s$. Let $s''=\interval{s}{y_0+1}{y_1}$ denote the interval obtained from $s'$ by removing the first 1. Then the following identities hold.
\begin{eqnarray*}
\txtlength(m') &=& \txtones(s') = \txtones(s'') +1 \\
\txtones(m') &=& \txtlength(s') - k\,\txtones(s') = \txtlength(s'') + 1 - k(\txtones(s'') +1)
\end{eqnarray*}

By substituting these and the identities obtained in Step 1 into (\ref{eqn:even->rec-bal 2}) we obtain 
\begin{eqnarray*}
\txtlength(s'')+1-k(\txtones(s'')+1) & \geq & \frac{\PP(s) - k\ones(s)}{\ones(s)}(\txtones(s'') + 1) + 1
\end{eqnarray*}
which by canceling terms implies $\txtlength(s'')  \geq  \frac{\PP(s)}{\ones(s)}(\txtones(s'') + 1)$ and therefore ${\txtones(s'') + 1} \leq \frac{\ones(s)}{\PP(s)}\txtlength(s'')$. This means that $s''$ is an interval in $s$ with too few~1s, contradicting the assumption that the 1s in $s$ are evenly distributed.
\end{proof}

Finally, we turn to the characterization that a periodic 0,1-sequence is Sturmian if and only if it is swap symmetric. In Section~\ref{sec:characterization} we have already tried to motivate that Sturmian sequences are swap symmetric, and the proof indeed proceeds as suggested by Figure~\ref{fig:examples-symmetry}.

\begin{proof}[Proof of Theorem~\ref{thm:equivalences}:~\ref{eq:sturmian} $\rAr$~\ref{eq:symmetry}]
Let $0<b<a$ and let $s=B_{a,b}=(|\col_n(C_{a,b})|)_n$. We claim that
\[
 \swap{s}{0}=(|\col_n(C_{a,b,-\frac{1}{a}})|)_n \equiv (|\col_n(C_{a,b})|)_n = s
\]
which completes the proof. The equivalence in the second step holds by Lemma~\ref{lem:dist-to-line}. All that is left to show is why the first equality holds.

To this end we argue as follows (see Figure~\ref{fig:examples-symmetry}): First we observe that shifting the line down by $-\frac{1}{a}$ only changes those columns $\col_x(S_{a,b})$ with $x\mod a=0$. More precisely $\col_x(S_{a,b,-\frac{1}{a}}) = \col_x(S_{a,b})-\twovec{0}{1}$ if $0=x\mod a$ and $\col_x(S_{a,b,-\frac{1}{a}})=\col_x(S_{a,b})$ otherwise. Now we observe that a point $v$ in a flat staircase is a corner if and only if $v-e_1$ is not in the staircase. As we know which columns changed, and that $a\geq2$, this allows us determine where the corners are after the shift. If $x\mod a= 0$, then $|\col_x(C_{a,b,-\frac{1}{a}})|=0$. If $x\mod a= 1$, then $|\col_x(C_{a,b,-\frac{1}{a}})|=1$. Otherwise $\col_x(C_{a,b,-\frac{1}{a}})=\col_x(C_{a,b})$.
\end{proof}

To show that swap symmetric sequences are Sturmian, we first prove two lemmas. Note that in Lemma~\ref{lem:symmetric->block-symmetric} we do not claim that if $s$ is swap symmetric, then $m(s)$ is balanced. However we can show that $m(s)$ is swap symmetric.\footnote{Our definition of swap symmetry was phrased such that it can be applied to arbitrary periodic sequences.} We then observe in Lemma~\ref{lem:symmetric->balanced} that swap symmetric sequences are necessarily balanced and hence $m(s)$ is balanced.

\begin{Lemma}
\label{lem:symmetric->block-symmetric}
Let $s$ be a periodic $0,1$-sequence. If $s$ is swap-symmetric and $\ones(s)>1$, then $m(s)$ is swap-symmetric.
\end{Lemma}

\begin{proof}
We know that $\PP(m(s))=\ones(s)>1.$ Let $0\leq i<\PP(s)$ be such that $\swap{s}{i}\equiv s$. Then $s_i=1$.  Say the block preceding $s_i$ is the $j$-th block of $s$. Swapping at $i$ makes all the $k$-th blocks of $s$ with $k \equiv j \mod \ones(s)$ larger by one and all the $k$-th blocks of $s$ with $k \equiv j+1 \mod \ones(s)$ smaller by one while leaving all other blocks unmodified. As $\PP(m(s))>1$, this means that $\swap{m(\swap{s}{i})}{j}\equiv m(s)$ but by assumption $\swap{s}{i}\equiv s$, so there exists an $l\in\ZZ$ such that
\[
  \swap{m(s)}{l} \equiv \swap{m(\swap{s}{i})}{j} \equiv m(s)
\]
which means that $m(s)$ is swap symmetric.
\end{proof}

\begin{Lemma}
\label{lem:symmetric->balanced}
If $s$ is a periodic swap-symmetric sequence, then $s$ is balanced.
\end{Lemma}

\begin{proof}
Assume to the contrary that $s$ contains at least three different entries. Let $0\leq i <\PP(s)$ be such that $\swap{s}{i}\equiv s$. Then $s_i\not=s_{i+1}$. Let $a=s_i$ and $b=s_{i+1}$. Let $c\in\ZZ$ be such that there exists a $j\in\ZZ$ with $s_j=c$ but $a\not= c\not=b$. We now define the parameter $d(s)$, which is the sum of the distances of any occurrence of $a$ in $\text{period}(s)$ to the closest preceding occurrence of $c$ in $s$, i.e.
\begin{eqnarray*}
d(s) &:=& \sum_{0\leq k < \PP(s), s_k=a} k - \max \set{l < k}{s_l=c}.
\end{eqnarray*}
Note that if $s\equiv s'$, then $d(s)=d(s')$ as shifting a sequence to the left or right does not affect the distances between occurrences of values. Now the swap at $i$ interchanges the $a$ at position $i$ and the $b$ at position $i+1$, which increases the distance of this occurrence of $a$ to the previous $c$ by $1$ and leaves all other distances of an occurrence of $a$ to a previous $c$ unaffected. Hence $d(\swap{s}{i})=d(s)+1$ and so $\swap{s}{i}\not\equiv s$, which is a contradiction.
\end{proof}

After these two lemmas, the proof that swap symmetric sequences are Sturmian is easy. Again we proceed by showing that swap symmetric sequences are recursively balanced.

\begin{proof}[Proof of Theorem~\ref{thm:equivalences}:~\ref{eq:symmetry} $\rAr$~\ref{eq:rec-bal}]
  Let $s$ be a periodic 0,1-sequence that is swap symmetric. There is an index $i$ at which we can swap, so $\ones(s)>0$. If $\ones(s)=1$, $s$ is recursively balanced by definition. So we can assume $\ones(s)>1$. By Lemma~\ref{lem:symmetric->block-symmetric} it follows that $m(s)$ is swap symmetric. By Lemma~\ref{lem:symmetric->balanced} it follows that $m(s)$ is balanced. Taking both together we conclude that $s$ is block balanced and that $\red{m}(s)$ is well defined and swap symmetric. By induction we infer that $\red{m}(s)$ is recursively balanced. But if $s$ is block balanced and $\red{m}(s)$ is recursively balanced, then $s$ is recursively balanced.
\end{proof}

%%% Local Variables: 
%%% mode: latex
%%% TeX-master: "Staircases_in_Z2"
%%% End: 

% \input{short_reps}
%%%%%%%%%%%%%%%%%%%%%%%%%%%%%%%%%%%%%
%                                   %
%             section 6             %
%                                   %
%%%%%%%%%%%%%%%%%%%%%%%%%%%%%%%%%%%%%

\section{Application: Short Representations}
\label{sec:short-representations}

It is a celebrated result by Barvinok \cite{Barvinok} that there is a polynomial time algorithm for counting the number of lattice points inside a given rational polytope when the dimension of the polytope is fixed. Note that if the dimension is an input variable, the problem gets $NP$-hard \cite{GJ79}. For more about the algorithm see \cite{DeL05}, \cite{LHRT04} and the textbook \cite{BarvinokZurich}.

The crucial ingredient of Barvinok's proof was his result that the set of lattice points in a simplicial cone of any fixed dimension can be expressed using a short generating function. In this section we give a new proof of this result for the special case of 2-dimensional cones (Theorem~\ref{thm:barvinok}). A generalization of our proof to higher dimensions is not immediate, we hope, however, that such a generalization can be found in the future.

We consider the Laurent polynomial ring $\KK[x_1^\pm,\ldots,x_d^\pm]$. For a vector $m=(m_1,\ldots,m_d)\in\ZZ^d$ we write $x^m:=x_1^{m_1}\ldots x_d^{m_d}$. This gives a bijection between $\ZZ^d$ and the set of monomials in $\KK[x_1^\pm,\ldots,x_d^\pm]$. We can thus represent the set of lattice points in a polyhedron $P$ by the generating function $f_P(x)=\sum_{m\in\ZZ^d\cap P}x^m$. If $P\cap\ZZ^d$ is large, this representation of $f_P$ contains many terms. Using rational functions it is possible to find shorter representations of $f_P$. For example the generating function of all non-negative integral multiples of a vector $m$ can be written as $\frac{1}{1-x^m}$, which allows us to express point sets like $\{0,m,\ldots,km\}$ as $\frac{1-x^{(k+1)m}}{1-x^m}$.

Developing these notions in detail is beyond the scope of this article. As references we recommend \cite{BeckRobins} and \cite{BarvinokZurich}. However, we would like to point out, informally, how the algebraic operations on generating functions correspond to geometric operations: The sum of generating functions corresponds to the union of the respective sets. The product of generating functions corresponds to the Minkowski sum of the respective sets. Taking the product of a generating function and a monomial $x^m$ thus corresponds to translation by $m$. Evaluating the generating function $f_P(x)$ at the values $x^{m_1},\ldots,x^{m_d}$ for $m_1,\ldots,m_d\in \ZZ^d$ corresponds to applying the linear map given by the matrix $A=(m_1 \ldots m_d)$ that has the $m_i$ as columns to the set $P$. 

As we already mentioned, for every 2-dimensional rational cone $K$ in $\RR^2$ there exists a lattice transform $A$ such that $AK=\cone\left(\twovec{1}{0}, \twovec{a}{b}\right)$ for $a,b\in\NN$ with $\gcd(a,b)=1$. Barvinok showed a general version of the following theorem for cones of any dimension. We are going to give a new proof of this version for 2-dimensional cones (recall the definitions from the end of Section~\ref{sec:geometricidea}).

\begin{Thm}
\label{thm:barvinok}
Let $a,b\in\NN$ with $\gcd(a,b)=1$. Let $K=\cone\left(\twovec{1}{0}, \twovec{a}{b}\right)$. 
Then $f_K$ admits a representation as a rational function with $\mathcal{O}(\log a)$ terms and this representation can be computed in time polynomial in $\log a + \log b$.
\end{Thm}

\begin{proof}
\emph{Step 1.} We express $f_K$ in terms of $f_{\Delta'_{a,b}}$. To this end we first note that 
$$\cone(\twovec{1}{0}, \twovec{a}{b}) = \bigcup_{k\geq 0} k\twovec{a}{b} + (\cone(\twovec{1}{0}, \twovec{a}{b}) \cap \RR\times[0,b) )$$
and
\begin{eqnarray*}
\ZZ^2 \cap \cone(\twovec{1}{0}, \twovec{a}{b}) \cap \RR\times[0,b) & =&  \{\twovec{0}{0}\} \cup T'_{a,b} \\ && \cup \left(\set{\twovec{i}{0}}{i\geq a+1} + \set{\twovec{0}{j}}{0\leq j\leq b-1}\right).
\end{eqnarray*}
See Figure~\ref{fig:cone-in-terms-of-triangle}. In terms of generating functions this translates into 
$$f_K(x)=\frac{1}{1-x_1^ax_2^b} \left( 1 + f_{\Delta'_{a,b}}(x) + \frac{x_1^{a+1}}{1-x_1}\frac{1-x_2^b}{1-x_2} \right).$$
Here we express $f_K$ using $f_{\Delta'_{a,b}}$ and a constant number of other terms. So it suffices to give a short expression of $f_{\Delta'_{a,b}}$.

\begin{figure}[ht]
\begin{center}
\includegraphics[scale=0.3]{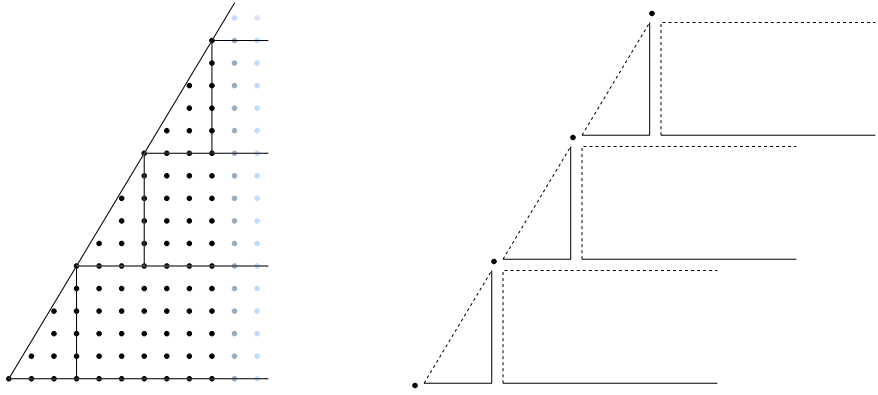}
\caption{\label{fig:cone-in-terms-of-triangle}Expressing the lattice points in a cone in terms of triangles. The right picture shows the occurring shapes (dashed lines indicate open faces).}
\end{center}
\end{figure}

\emph{Step 2.} We use the recursion from Lemma~\ref{lem:recursion-triangle} to give a short expression for $f_{\Delta'_{a,b}}$. Let $(c_n)_n$ be the sequence defined by $c_1=b$, $c_2=a$ and $c_{i+2}=c_{i} \mod c_{i+1}$ and let $j$ be the index such that $c_{j+1}=1$ and $c_{j+2}=0$. We express $f_{T'_{c_{i+1},c_{i}}}$ in terms of $f_{T'_{c_{i+2},c_{i+1}}}$, by applying first~\ref{lem:recursion-triangle}.1 and then~\ref{lem:recursion-triangle}.2 and~\ref{lem:recursion-triangle}.4.
\begin{eqnarray*}
f_{T'_{c_{i+1},c_{i}}}(x_1,x_2) &=& f_{T'_{c_{i+1},c_{i+2}}}(x_1x_2^{c_{i} \div c_{i+1}},x_2) \\&& + f_{T'_{c_{i+1},(c_{i}\div c_{i+1})c_{i+1}}}(x_1,x_2) \\ &=&
x_1^{c_{i+1}}x_2^{c_{i+2}}\cdot f_{T'_{c_{i+2},c_{i+1}}}(x_2^{-1},x_1^{-1}x_2^{-(c_{i} \div c_{i+1})}) \\
&& + \frac{1}{1-x_2} \cdot \left( \frac{1-x_1^{c_{i+1}+1}}{1-x_1} - \frac{1-x_1^{c_{i+1}+1}x_2^{(c_{i}\div c_{i+1})(c_{i+1}+1)}}{1-x_1x_2^{c_{i}\div c_{i+1}}} \right).
\end{eqnarray*}
We have thus expressed $f_{T'_{c_{i+1},c_{i}}}$ using a constant number of other terms. We proceed in this fashion until we reach the case $f_{T'_{c_{j+1},c_{j}}}=f_{T'_{1,c_{j}}}$ which we can solve directly using~\ref{lem:recursion-triangle}.3:
$$f_{T'_{1,c_{j}}}=x_1\frac{1-x_2^{c_j}}{1-x_2}.$$

\emph{Step 3.} The expression is short and can be computed in polynomial time as the Euclidean Algorithm is fast. By Lemma~\ref{lem:euclidean-algorithm} the number of iterations required in step 2 is $\mathcal{O}(\log a)$. In each step we pick up a constant number of terms. So the total number of terms in the final expression is $\mathcal{O}(\log a)$. The algorithm runs in time polynomial in $\log a + \log b$ as the numbers $c_{i+2}=c_i \mod c_{i+1}$ and $c_i \div c_{i+1}$ can be computed in time polynomial in $\log c_i + \log c_{i+1}$.
\end{proof}

This proof of Theorem~\ref{thm:barvinok} differs from Barvinok's. Barvinok gives a \emph{signed} decomposition of a cone into \emph{unimodular} cones. We give a \emph{positive} decomposition of the triangle $T'_{a,b}$ into triangles $T'_{c_{i+1},(c_{i}\div c_{i+1})c_{i+1}}$ that are not unimodular but easy to describe, i.e.~using a constant number of terms. 

In this context ``positive'' means that the 2-dimensional triangle $T'_{a,b}$ is written as a disjoint union of half-open 2\nolinebreak-dimensional triangles $T'_{a,b}$. This does not mean that the numerator of the rational function has only positive coefficients. Negative coefficients appear in the ``easy'' description of the triangles $T'_{c_{i+1},(c_{i}\div c_{i+1})c_{i+1}}$. 

\vspace{1em}

Lemma~\ref{lem:recursion-parallelepiped} can be used to obtain a short representation of the generating function of the lattice points in the fundamental parallelepiped of any rational cone in the plane. We implement this idea in the proof of Theorem~\ref{thm:dedekind-carlitz} in Section~\ref{sec:dedekind-carlitz}. This representation can also be used to give an alternative proof of Theorem~\ref{thm:barvinok}. Again the representation is positive in the sense that the set of lattice points in the fundamental parallelepiped is expressed as a disjoint union of Minkowski sums of intervals. But of course still negative coefficients appear as they appear in the representation of intervals. As opposed to the representation based on triangles, the representation based on fundamental parallelepipeds relies on taking products; so with this approach expanding the products in the numerators leads to an expression that is not short any more.

\vspace{1em}

It is also possible to give a recursion similar to~\ref{lem:recursion}, \ref{lem:recursion-parallelepiped} and~\ref{lem:recursion-triangle} directly for cones. However, in this case the recursion does require us to take differences of sets and we do not obtain a ``positive'' decomposition. Nonetheless the recursion differs from the one based on the continued fraction expansion of $\frac{b}{a}$ given in~\cite[Chapter~15]{BarvinokZurich}.

%%% Local Variables: 
%%% mode: latex
%%% TeX-master: "Staircases_in_Z2"
%%% End: 

\section{Application: Dedekind-Carlitz Polynomials}
\label{sec:dedekind-carlitz}

Given $0<a,b\in\NN$ with $\gcd(a,b)=1$, Carlitz introduced the following polynomial generalization of Dedekind sums, which Beck, Haase and Matthews in \cite{BeckHaase08} call the \defn{Dedekind-Carlitz polynomial}:
\[
 c_{a,b}(x,y):=\sum_{k=1}^{a-1}x^{k-1}y^{\floor{\frac{b}{a}k}}.
\]
For a brief overview of the history of and literature about Dedekind sums and the Dedekind-Carlitz polynomial, we refer to \cite{BeckHaase08}. There also the relationship between Dedekind-Carlitz polynomials and the fundamental parallelepipeds of cones (see below) is established. Appealing to Barvinok's Theorem, Beck, Haase and Matthews conclude that the Dedekind-Carlitz polynomial can be computed in polynomial time and must have a short representation,\footnote{In \cite{BeckHaase08} this is argued even for higher-dimensional Dedekind-Carlitz polynomials.} however they do not give such a short representation explicitly. Also they remark that Dedekind sums can be computed efficiently in the style of the Euclidean Algorithm and ask if such a recursive procedure also exists for Dedekind-Carlitz polynomials. In this section we use the recursion for the lattice points inside a fundamental parallelepiped developed in Section~\ref{sec:geometricidea} to give an explicit recursion formula that allows one to compute short representations of Dedekind-Carlitz polynomials in the style of the Euclidean Algorithm.

We first observe that $c_{a,b}$ is the generating function of the set
\[
 \set{z\in\ZZ^2}{z_1=\floor{\frac{b}{a}z_2}, 1\leq z_1 \leq a-1} - \mtwovec{1}{0} = \pS{a}{b} - \mtwovec{1}{0}.
\]
which is just a translate of the set of lattice points in the open fundamental pa\-\mbox{rall}elepiped $\pS{a}{b}$. Hence the recursion given in Lemma~\ref{lem:recursion-parallelepiped} can be used to give a recursion formula for Dedekind-Carlitz sums in the spirit of the Euclidean Algorithm. To this end, we use $\gD{a}{b}(x,y)$ and $\gR{a}{b}(x,y)$ to denote the generating functions of the sets $\pS{a}{b}$ and $\pC{a}{b}$, respectively. So
\begin{eqnarray*}
\gD{a}{b}(x,y) & = & \sum_{k=1}^{a-1}x^{k}y^{\floor{\frac{b}{a}k}}, \\
\gR{a}{b}(x,y) & = & \sum_{k=1}^{b-1}x^{\ceil{\frac{a}{b}k}}y^k.
\end{eqnarray*}
Now, by simply translating the geometric operations into the language of generating functions, we obtain the following theorem. In \cite{BeckHaase08} this result is derived from Barvinok's Theorem. We give an explicit recursion formula in the proof.

\begin{Thm}
\label{thm:dedekind-carlitz}
Let $0<a,b\in\NN$ with $\gcd(a,b)=1$. Then $c_{a,b}$ admits a representation as a rational function with $\mathcal{O}(\log a)$ terms and this representation can be computed in time polynomial in $\log a + \log b$.
\end{Thm}

As was said before, the representation we obtain is ``positive'' in the sense that we build a partition of the set $\pS{a}{b}$ using Minkowski sums and disjoint unions of intervals. It is not positive in the sense that all coefficients appearing the representation are positive, as the representations of intervals that we use contain coefficients with opposite signs.

It is important to stress that the representation we obtain makes heavy use of Minkowski sums of intervals. In the language of generating functions, this corresponds to taking products of expressions of the form $\frac{1-x^{ku}}{1-x^u}$ for $k\in\NN$ and $u\in\ZZ^2$. Expanding the numerators of these products by applying the distributive law may lead to a numerator with a number of summands exponential in the number of factors of the product. So the expression we obtain is only short, if products are not expanded. We note that this problem does not occur with the representation we used in the proof of Theorem~\ref{thm:barvinok}.

\begin{proof}
First we note that $c_{a,b}(x,y)=x^{-1}\gD{a}{b}(x,y)$. Now we construct a short representation of $\gD{a}{b}(x,y)$ by applying Lemma~\ref{lem:recursion-parallelepiped} inductively. To that end let $(c_n)_n$ be the sequence defined by $c_1=b$, $c_2=a$ and $c_{i+2}=c_{i} \mod c_{i+1}$ and let $j$ be the index such that $c_{j+1}=1$ and $c_{j+2}=0$. Such a $j$ exists because $\gcd(a,b)=1$.

By Lemma~\ref{lem:recursion-parallelepiped}.2 we can assume without loss of generality $c_1>c_2$. Then for all $i\geq 1$
\begin{eqnarray*}
\gD{c_{i+1}}{c_{i}}(x,y) & = & \gD{c_{i+1}}{c_{i+2}}(xy^{c_i \div c_{i+1}},y) \\
\gR{c_{i+1}}{c_i}(x,y) &=& \frac{1-y^{-(c_i \div c_{i+1})}}{1-y^{-1}}\gD{c_{i+1},c_{i+2}}(xy^{c_i \div c_{i+1}},y) \\
&& + y^{-(c_i \div c_{i+1})}\gR{c_{i+1},c_{i+2}}(xy^{c_i \div c_{i+1}},y) \\
&& + x^ay^b \frac{y^{-1}-y^{-(c_i \div c_{i+1})}}{1-y^{-1}}
\end{eqnarray*}
and
\begin{eqnarray*}
\gR{c_{i}}{c_{i+1}}(x,y)=x^ay^b\gD{c_{i+1}}{c_i}(-y,-x) & \text{ and } & 
\gD{c_{i}}{c_{i+1}}(x,y)=x^ay^b\gR{c_{i+1}}{c_i}(-y,-x).
\end{eqnarray*}

Together with
\begin{eqnarray*}
\gR{c_j}{c_{j+1}}(x,y)=0 & \text{ and } & \gD{c_j}{c_{j+1}}(x,y) = \frac{x-x^a}{1-x}
\end{eqnarray*}
this gives us a recursion formula for $\gD{a}{b}(x,y)$. In each step of the recursion we pick up only a constant number of terms and by Lemma~\ref{lem:euclidean-algorithm} we need only $\mathcal{O}(\log a)$ steps, so the resulting representation has only $\mathcal{O}(\log a)$ terms. As standard arithmetic operations can be computed in time polynomial in the input length, the algorithm runs in time polynomial in $\log a+\log b$.
\end{proof}

Note that by using products, one can give a representation of the lattice points in an interval of length $n$ with $\mathcal{O}(\log n)$ many terms and \emph{without using rational functions}. Using such a representation, the above proof gives a representation with $\mathcal{O}(\log^2 a)$ terms in time polynomial in $\log a + \log b$. Moreover this representation then is positive in that every coefficient appearing in this expression has a positive sign.

%%% Local Variables: 
%%% mode: latex
%%% TeX-master: "Staircases_in_Z2"
%%% End: 
%%%%%%%%%%%%%%%%%%%%%%%%%%%%%%%%%%%%%
%                                   %
%             section 7             %
%                                   %
%%%%%%%%%%%%%%%%%%%%%%%%%%%%%%%%%%%%%
\section{Application: Theorem of White}\label{whitethm}

To conclude this paper, we give a partly new proof for a theorem of White \cite[pp.390-394]{White64}, characterizing lattices tetrahedra containing no lattice points but the vertices. Several proofs appeared over the years, e.g. by Noordzij \cite{Noordzij81}, Scarf \cite{Scarf85} (based partly on work by R. Howe) and recently Reznick \cite{Reznick06}, who also gives an overview of the history of this theorem. Furthermore his proof has the advantage that it keeps track of the vertices, at the cost of geometric transparency. We construct our proof based on ideas in \cite{Scarf85} and mainly \cite{Reznick06}.

\vspace{1em}
For $(a,b,n)\in \ZZ^3$ we define the tetrahedron $T_{a,b,n}$ as
\[
 \conv\left\{\mthreevec{0}{0}{0}, \mthreevec{1}{0}{0}, \mthreevec{0}{1}{0}, \mthreevec{a}{b}{n}\right\}.
\]

This should not be confused with the point set $T_{a,b}$ defined in Section~\ref{sec:geometricidea}. As we will not use the latter any more, no ambiguities should arise.

A ``hidden'' parameter, as Reznick writes, is $c=n-a-b+1$ (although he considers a slightly different $c$). We will see that $c$ plays a role equal to the ones of $a$ and $b$ in $T_{a,b,n}$. Also note that $a+b+c=n+1$.

\vspace{1em}
We call two tetrahedra $T$ and $T'$ \defn{equivalent} ($T\approx T'$), if there is an affine lattice transformation which takes the vertices of $T$ to the vertices of $T'$. 

\vspace{1em}
A lattice simplex $T$ is \defn{clean} if there are no non-vertex lattice points on the boundary. If there are also no lattice points in the interior of $T$ (i.e. the vertices are the only lattice points), then we call $T$ \defn{empty}.

%%%%%%%%%%%%%%%%%%%%%%%%%%%%%%%%%%%%%
%           Thm of White            %
%%%%%%%%%%%%%%%%%%%%%%%%%%%%%%%%%%%%%
\begin{Thm}[White]\label{thmofwhite}
A lattice tetrahedron $T$ is empty if and only if it is equivalent to $T_{0,0,1}$ or to some $T_{1,d,n}$, where $\gcd(d,n)=1$ and $1\leq d\leq n-1$.
\end{Thm}

%%% begin proof %%%%%%%%%%%%%%%%%%%%%%%%%%%%%%%%%%%%%%%%%%%%%%%%%%%%%%%%%%%%%%%%%%%%%%%%%%%%%%%%%%%%%%%
\begin{proof}[Proof of Theorem~\ref{thmofwhite} (Necessity)]
As we easily see, $T_{0,0,1}$ is empty. So we consider  $T_{1,d,n}$, where $\gcd(d,n)=1$ and $1\leq d\leq n-1$. Let $w\in\ZZ^3\cap T_{1,d,n}$. As the first coordinate of all vertices of $T_{1,d,n}$ is either $0$ or $1$, we know $w_1\in\{0,1\}$.

If $w_1=0$, then $w\in\conv\left\{\threevec{0}{0}{0},\threevec{0}{1}{0}\right\}$ and $w$ is a vertex.

If $w_1=1$, then $w\in\conv\left\{\threevec{1}{0}{0},\threevec{1}{d}{n}\right\}$ . As $\gcd(d,n)=1$ the vector $\threevec{0}{d}{n}$ is primitive and so $w$ is again a vertex.

Therefore $T_{1,d,n}$ and any equivalent tetrahedron is empty. This proves that lattice tetrahedra of the form $T_{0,0,1}$ or $T_{1,d,n}$ with  $\gcd(d,n)=1$ and $1\leq d\leq n-1$ are necessarily empty.
\end{proof}
%%% end proof %%%%%%%%%%%%%%%%%%%%%%%%%%%%%%%%%%%%%%%%%%%%%%%%%%%%%%%%%%%%%%%%%%%%%%%%%%%%%%%%%%%%%%%
To show the sufficiency, we use the following theorem by Reeve. A nice proof can be found in \cite[pp.5-6]{Reznick06}.
%%%%%%%%%%%%%%%%%%%%%%%%%%%%%%%%%%%%%
%           Thm of Reeve            %
%%%%%%%%%%%%%%%%%%%%%%%%%%%%%%%%%%%%%
\begin{Thm}[\cite{Reeve57}]\label{thmofreeve}
 The lattice tetrahedron $T$ is clean if and only if $T\approx T_{0,0,1}$ or $T\approx T_{a,b,n}$, where 
\[
 n\geq 2,\quad 0 \leq a,b\leq n-1 \quad\text{ and }\quad \gcd(a,n)=\gcd(b,n)=\gcd(1-a-b,n)=1.
\]
\end{Thm}
We will now prove the sufficiency to motivate the rest of the section, where we anticipate the results that are stated and proved only afterwards.

\begin{proof}[Proof of Theorem~\ref{thmofwhite} (Sufficiency)]
If $T$ is an empty lattice tetrahedron, then in particular it is clean and thus by Theorem~\ref{thmofreeve} equivalent to $T_{0,0,1}$ or some $T_{a,b,n}$. If $T\approx T_{0,0,1}$, we are done. Otherwise $T_{a,b,n}\approx T$ is of course also empty and therefore fulfills the conditions for Lemma~\ref{scarf3.2}.

\vspace{0.5em}
This in turn enables us to use Lemma~\ref{F&F} which tells us that one of $a,b,c$ equals $1$. Finally, we can see by Lemma~\ref{scarf3.1} that we can choose the order of the coordinates freely, so we get $T\approx T_{1,d,n}$, where $\gcd(d,n)=1$ and $1\leq d\leq n-1$, again by Theorem~\ref{thmofreeve}.
\end{proof}

It turns out to be useful to describe a clean tetrahedron $T_{a,b,n}$ in a slightly different way:

%%%%%%%%%%%%%%%%%%%%%%%%%%%%%%%%%%%%%
%         Lemma 3.1 in Scarf        %
%%%%%%%%%%%%%%%%%%%%%%%%%%%%%%%%%%%%%
\begin{Lemma}\label{scarf3.1}
Let $T_{a,b,n}$ be empty and $0 \leq a,b\leq n-1$. Then 
\[
 T_{a,b,n}\approx \conv\left\{\mthreevec{1}{0}{0}, \mthreevec{0}{1}{0}, \mthreevec{0}{0}{1}, \mthreevec{a}{b}{c}\right\},
\quad\text{ and } c\geq 1.
\]
\end{Lemma}

In the next two proofs we follow mostly \cite[pp.411f]{Scarf85}.
%%% begin proof %%%%%%%%%%%%%%%%%%%%%%%%%%%%%%%%%%%%%%%%%%%%%%%%%%%%%%%%%%%%%%%%%%%%%%%%%%%%%%%%%%%%%%%
\begin{proof} 
If $n<a+b$, then 
\[
 \mthreevec{1}{1}{1}=\alpha_1\mthreevec{0}{0}{0}+ \alpha_2 \mthreevec{1}{0}{0}+ \alpha_3 \mthreevec{0}{1}{0}+ \alpha_4 \mthreevec{a}{b}{n}, \text{ where }
\]
$\alpha_4=\frac{1}{n},\quad\alpha_3=1-\frac{b}{n},\quad\alpha_2=1-\frac{a}{n},$ and $\alpha_1=1-\alpha_2-\alpha_3-\alpha_4= \frac{a+b-n-1}{n}$. But this means that $0\leq \alpha_1<1$ and $0<\alpha_2,\alpha_3,\alpha_4<1$, and thus $\threevec{1}{1}{1}\in T_{a,b,n}$, contradicting the assumption that $T_{a,b,n}$ is empty.

So we know that $n\geq a+b$, which proves $c\geq 1$. The affine lattice transformation
\[
x \mapsto 
\begin{pmatrix}
1 & 0 & 0 \\
0 & 1 & 0 \\
-1 & -1 & 1
\end{pmatrix}
x +
\begin{pmatrix}0\\0\\1\end{pmatrix}
\]
gives us the desired form for $T_{a,b,n}$.
\end{proof}
%%% end proof %%%%%%%%%%%%%%%%%%%%%%%%%%%%%%%%%%%%%%%%%%%%%%%%%%%%%%%%%%%%%%%%%%%%%%%%%%%%%%%%%%%%%%%
The key ingredient for the rest of the proof of Theorem~\ref{thmofwhite} is to look at the Beatty sequences for $\frac{a}{n}$, $\frac{b}{n}$ and $\frac{c}{n}$ simultaneously. For this purpose let us define the sum of the sequences, i.e.
\begin{eqnarray*}
 f(k) & \define & B_{n,a}(k) + B_{n,b}(k) + B_{n,c}(k) \\ 
&=& \floor{\frac{a}{n}k}+\floor{\frac{b}{n}k}+\floor{\frac{c}{n}k}-\floor{\frac{a}{n}(k-1)}-\floor{\frac{b}{n}(k-1)} -\floor{\frac{c}{n}(k-1)}.
\end{eqnarray*}
This function has a strong connection to Theorem~\ref{thmofwhite} as we see next.

%%%%%%%%%%%%%%%%%%%%%%%%%%%%%%%%%%%%%
%         Lemma 3.2 in Scarf        %
%%%%%%%%%%%%%%%%%%%%%%%%%%%%%%%%%%%%%
\begin{Lemma}\label{scarf3.2}
If $T_{a,b,n}$ is empty, then $f(k)=1$ for $k=2,\ldots,n-1$.
\end{Lemma}

%%% begin proof %%%%%%%%%%%%%%%%%%%%%%%%%%%%%%%%%%%%%%%%%%%%%%%%%%%%%%%%%%%%%%%%%%%%%%%%%%%%%%%%%%%%%%%
\begin{proof}
Let us first define the function $g(k)\define \ceil{\frac{a}{n}k}+\ceil{\frac{b}{n}k}+\ceil{\frac{c}{n}k}$. An easy computation verifies that for $a,b,c$ relatively prime to $n$ and $2\leq k\leq n-1$, we have $f(k)=g(k)-3-(g(k-1)-3)=g(k)-g(k-1)$. 

We will now show that $g(k)=k+2$ for $k=1,\ldots,n-1$ and also that $a,b,c$ are relatively prime to $n$.

\vspace{0.5em}
Suppose that $g(k)\leq k+1$ for some $k$. Then we can define a lattice point $\threevec{p}{q}{r}$ with
\[
 p\geq \ceil{\frac{a}{n}k},\quad q\geq \ceil{\frac{b}{n}k},\quad r\geq\ceil{\frac{c}{n}k},
\]
and $p+q+r=k+1$. But then we find $\alpha_1,\ldots,\alpha_4$ with 
\[
 \begin{pmatrix}
  p\\
  q\\
  r
 \end{pmatrix}=
\begin{pmatrix}
 1 & 0 & 0 & a\\
 0 & 1 & 0 & b\\
 0 & 0 & 1 & c
\end{pmatrix}
\;\begin{pmatrix}
   \alpha_1\\
\vdots\\
\alpha_4
  \end{pmatrix}, \quad \text{ where }
\]
$\alpha_1=p-\frac{a}{n}k,\quad\alpha_2=q-\frac{b}{n}k,\quad\alpha_3=r-\frac{c}{n}k,\quad\alpha_4=\frac{1}{n}k\quad$ and 
\begin{eqnarray*}
 \sum_{i=1}^4 \alpha_i&=& p+q+r-\frac{a}{n}k-\frac{b}{n}k-\frac{c}{n}k+\frac{k}{n}\\
&=&k+1-\frac{(a+b+c-1)k}{n}\quad=\quad k+1-k\\
&=& 1.
\end{eqnarray*}
As this means that $\left(\begin{smallmatrix}p\\q\\r\end{smallmatrix}\right)$ is in $\conv\left\{\threevec{1}{0}{0}, \threevec{0}{1}{0}, \threevec{0}{0}{1},\threevec{a}{b}{c}\right\}\approx T_{a,b,n}$, we have a contradiction.

\vspace{0.5em}
So we know that $g(k)\geq k+2$.

If one of $a,b,c$ is not relatively prime to $n$, say it is $c$, then $\frac{c}{n}k\in \ZZ$ for some $k\in\{1,\ldots,n-1\}$. So we get 
\begin{eqnarray*}
 g(k)+g(n-k)&=&\ceil{\frac{a}{n}k}+\ceil{\frac{b}{n}k}+\ceil{\frac{c}{n}k}+ \ceil{a-\frac{a}{n}k}+\ceil{b-\frac{b}{n}k}+\ceil{c-\frac{c}{n}k}\\
&=&\ceil{\frac{a}{n}k}+\ceil{\frac{b}{n}k}+\ceil{\frac{c}{n}k}+ a-\floor{\frac{a}{n}k}+b-\floor{\frac{b}{n}k}+c-\floor{\frac{c}{n}k},
\end{eqnarray*}
and by the assumption on $c$ this is at most $a+b+c+2=n+3$.
But then either $g(k)\leq k+1$ or $g(n-k)\leq n-k+1$, which cannot be true for an empty $T_{a,b,n}$.

\vspace{0.5em}
If they are all relatively prime to $n$, then $g(k)+g(n-k)=a+b+c+3= n+4$.
Together with $g(k)\geq k+2$ and $g(n-k)\geq n-k+2$ we get the desired equality.

\vspace{0.5em}
We have now seen that $g(k)=k+2$ for $k=1,\ldots,n-1$. Together with $f(k)=g(k)-g(k-1)$ this implies $f(k)=1$ for $k=2,\ldots,n-1$.
\end{proof}
%%% end proof %%%%%%%%%%%%%%%%%%%%%%%%%%%%%%%%%%%%%%%%%%%%%%%%%%%%%%%%%%%%%%%%%%%%%%%%%%%%%%%%%%%%%%%

This is all we need to finish the proof:
%%%%%%%%%%%%%%%%%%%%%%%%%%%%%%%%%%%%%
%      Lemma of Felix and Fred      %
%%%%%%%%%%%%%%%%%%%%%%%%%%%%%%%%%%%%%
\begin{Lemma}\label{F&F}
 If $f(k)=1$ for $k=2,\ldots,n-1$, then at least one of $a,b$ or $c$ equals $1$.
\end{Lemma}

This is the new part of the proof. It builds onto the observations about Sturmian sequences developed in the first half of this article. In particular we make use of the fact that Sturmian sequences are block balanced, and, more generally, that the 1s in a Sturmian sequence are evenly distributed. See Theorem~\ref{thm:equivalences}\ref{eq:even}.

%%% begin proof %%%%%%%%%%%%%%%%%%%%%%%%%%%%%%%%%%%%%%%%%%%%%%%%%%%%%%%%%%%%%%%%%%%%%%%%%%%%%%%%%%%%%%%
\begin{proof}
Suppose not. Without loss of generality $c<a,b$. We consider the intervals $\interval{B_{n,a}}{1}{n},\interval{B_{n,b}}{1}{n}$ and $\interval{B_{n,c}}{1}{n}$. As $c\geq 2$, there is another 1 in $\interval{B_{n,c}}{1}{n}$ apart from $B_{n,c}(n)=1$. Let $m$ be the position of the 1 preceding $B_{n,c}(n)=1$, i.e.\ $1< m < n$ such that $B_{n,c}(m)=1$ and $B_{n,c}(k)=0$ for $m<k<n$. Because $f(k)=1$ for all $2\leq k\leq n-1$, we know that $B_{n,a}(m)=B_{n,b}(m)=0$.

At this point we make a table listing the values of the three intervals at the positions $1,\ldots,n$,  filling in the values that we know and marking the values that we have not yet determined with $*$.

\[
 \begin{matrix}
&1& 2 & \ldots & m-1 & m & m+1 & \ldots & n-1 & n \\
&&&&&&&&&\\
\interval{B_{n,a}}{1}{n}=&0&*&\ldots&*&0&*&\ldots&*&1\\
&&&&&&&&&\\
\interval{B_{n,b}}{1}{n}=&0&*&\ldots&*&0&*&\ldots&*&1\\
&&&&&&&&&\\
\interval{B_{n,c}}{1}{n}=&0&*&\ldots&*&1&0&\ldots&0&1
 \end{matrix}
\]

Because $f(n-1)=1$ we know that either $\red{B}_{n,a}(n-1)=1$ or $\red{B}_{n,b}(n-1)=1$ and we may assume it is $\red{B}_{n,a}(n-1)=1$. We are now going to apply the following argument over and over again. By Theorem~\ref{thm:equivalences}\ref{eq:even} we know that if we find an interval of length $l$ in a Sturmian sequence that contains $t$ 1s, then any other interval of length $l$ in the same sequence has to contain at least $t-1$ and at most $t+1$ 1s. In this case $\txtones(\interval{B_{n,a}}{n-1}{n})=2$ and so both $\txtones(\interval{B_{n,a}}{m-1}{m})\geq 1$ and  $\txtones(\interval{B_{n,a}}{m}{m+1})\geq 1$, which means $B_{n,a}(m-1)=B_{n,a}(m+1)=1$ and consequently $B_{n,b}(m-1)=B_{n,b}(m+1)=0$. Now our table looks as follows.

\[
 \begin{matrix}
&1& 2 & \ldots\,\;\; &  m-1 & m & m+1 & \ldots & n-1 & n \\
&&&&&&&&&\\
\interval{B_{n,a}}{1}{n}=&0&*&\ldots\,*&1&0&1&*\,\ldots\,*&1&1\\
&&&&&&&&&\\
\interval{B_{n,b}}{1}{n}=&0&*&\ldots\,*&0&0&0&*\,\ldots\,*&0&1\\
&&&&&&&&&\\
\interval{B_{n,c}}{1}{n}=&0&*&\ldots\,*&0&1&0&\ldots&0&1
 \end{matrix}
\]

But now $\txtones(\interval{B_{n,b}}{m-1}{m+1})=0$ and so $B_{n,b}(n-2)=0$ and $B_{n,a}(n-2)=1$. This gives $\txtones(\interval{B_{n,a}}{n-2}{n})=3$ which allows us to deduce $B_{n,a}(m+2)=1$ and $B_{n,b}(m+2)=0$. Then we have $\txtones(\interval{B_{n,b}}{m-1}{m+2})=0$ and so $B_{n,b}(n-3)=0$ and $B_{n,a}(n-3)=1$, which gives $\txtones(\interval{B_{n,a}}{n-3}{n})=4$ and so $B_{n,a}(m+3)=1$ and $B_{n,b}(m+3)=0$. We continue this argument inductively until we have shown that $B_{n,b}(k)=0$ for $m+1\leq k \leq n-1$. 

At this point both $\interval{B_{n,b}}{m-1}{n-1}$ and $\interval{B_{n,c}}{m+1}{n-1}$ are intervals of consecutive 0s, of length $n-m+1$ and $n-m-1$, respectively, where the latter is maximal. So the blocks of $B_{n,b}$ are strictly larger than the blocks of $B_{n,c}$. As the 1s in Sturmian sequences are evenly distributed, this implies $c=\ones(B_{n,c}) > \ones(B_{n,b})= b$, in contradiction to our assumption.
\end{proof}
%%% end proof %%%%%%%%%%%%%%%%%%%%%%%%%%%%%%%%%%%%%%%%%%%%%%%%%%%%%%%%%%%%%%%%%%%%%%%%%%%%%%%%%%%%%%%

%%% Local Variables: 
%%% mode: latex
%%% TeX-master: "Staircases_in_Z2"
%%% End: 

{\small
{\bf Acknowledgements.} We would like to thank Christian Haase for drawing our attention to White's Theorem and asking if our observations about Sturmian sequences can be used to find a simpler proof. We thank Victor Alvarez for pointing us to the work of Balza-Gomez et al. We also thank Matthias Beck for helpful discussions about our results and getting us interested in lattice points inside triangles in the first place.
}

% \newpage
\bibliographystyle{amsalpha}
\bibliography{literatur}
\end{document}